\documentclass[11pt, reqno]{amsart}
\usepackage{amsmath}
\usepackage{mathrsfs}
\usepackage{yhmath}
\usepackage{amsxtra}
\usepackage{amscd}
\usepackage{amsthm}
\usepackage{amsfonts}
\usepackage{amssymb}
\usepackage{anysize}
\usepackage{bbding}
\usepackage{bbm}
\usepackage[bookmarks=true]{hyperref}
\usepackage{empheq}
\usepackage{enumerate}
\usepackage{eucal}
\usepackage{fancyhdr}
\usepackage{mathtools,graphics,latexsym}
\usepackage{pifont}
\usepackage{pmgraph}
\usepackage{stmaryrd}
\usepackage[usenames,dvipsnames]{color}
\usepackage{tikz-cd}

\marginsize{3cm}{3cm}{2.2cm}{2.2cm}

\input prepictex
\input pictex
\input postpictex
\SetSymbolFont{stmry}{bold}{U}{stmry}{m}{n}

\newtheorem{thm}{Theorem}[section]
\newtheorem{cor}[thm]{Corollary}
\newtheorem{lem}[thm]{Lemma}
\newtheorem{prop}[thm]{Proposition}
\newtheorem{conj}[thm]{Conjecture}

\theoremstyle{definition}

\theoremstyle{remark}
\newtheorem{rem}[thm]{Remark}

\numberwithin{equation}{section}

\begin{document}

\newcommand{\thmref}[1]{Theorem~\ref{#1}}
\newcommand{\thmcref}[1]{Theorem*~\ref{#1}}
\newcommand{\secref}[1]{Section~\ref{#1}}
\newcommand{\lemref}[1]{Lemma~\ref{#1}}
\newcommand{\propref}[1]{Proposition~\ref{#1}}
\newcommand{\propcref}[1]{Proposition*~\ref{#1}}
\newcommand{\corref}[1]{Corollary~\ref{#1}}
\newcommand{\conjref}[1]{Conjecture~\ref{#1}}
\newcommand{\remref}[1]{Remark~\ref{#1}}
\newcommand{\eqnref}[1]{(\ref{#1})}
\newcommand{\exref}[1]{Example~\ref{#1}}

\DeclarePairedDelimiterX\setc[2]{\{}{\}}{\,#1 \;\delimsize\vert\; #2\,}
\newcommand{\bn}[1]{\underline{#1\mkern-4mu}\mkern4mu }

\newcommand{\nc}{\newcommand}

\nc{\Z}{{\mathbb Z}}
\nc{\Zp}{\Z_+}
\nc{\C}{{\mathbb C}}
\nc{\N}{{\mathbb N}}
\nc{\F}{{\mathcal F}}

\nc{\bi}{\bibitem}
\nc{\wt}{\widetilde}
\nc{\wh}{\widehat}
\nc{\dpr}{{\prime \prime}}
\nc{\ov}{\overline}
\nc{\bd}{\boldsymbol}
\nc{\un}{\underline}

\nc{\al}{\alpha}
\nc{\be}{\beta}
\nc{\ga}{\gamma}
\nc{\de}{\delta}
\nc{\ep}{\epsilon}
\nc{\vep}{\varepsilon}
\nc{\La}{\Lambda}
\nc{\la}{\lambda}
\nc{\si}{\sigma}
\nc{\Sig}{{\bd \Sigma}}
\nc{\bog}{{\bd \omega}}

\nc{\mf}{\mathfrak}
\nc{\mc}{\mathcal}
\nc{\cA}{\mc{A}}
\nc{\cB}{\mc{B}}
\nc{\cC}{\mc{C}}
\nc{\ocA}{\ov{\mc{A}}}
\nc{\cN}{\mc{N}}
\nc{\cD}{\mc{D}}
\nc{\fB}{{\mf B}}
\nc{\fb}{{\mf b}}
\nc{\fg}{{\mf g}}
\nc{\fh}{{\mf h}}
\nc{\fL}{\mf{L}}
\nc{\fS}{\mf{S}}
\nc{\cS}{\mc{S}}

\nc{\s}{{\bf s}}
\nc{\w}{{\bf w}}
\nc{\x}{{\bf x}}
\nc{\y}{{\bf y}}
\nc{\z}{{\bf z}}
\nc{\et}{{\bd \eta}}
\nc{\zet}{{\bd \zeta}}

\nc{\pa}{\partial}
\nc{\pz}{\partial_z}
\nc{\Ber}{{\rm Ber}\,}
\nc{\Bers}{{\rm Ber}^{\bf s}\,}
\nc{\cdet}{{\rm cdet}\,}
\nc{\rdet}{{\rm rdet}\,}

\nc{\gl}{{\mf{gl}}}
\nc{\I}{\mathbb{I}}

\nc{\bs}{\bar{s}}
\nc{\bz}{\bar{0}}
\nc{\bo}{\bar{1}}
\nc{\hs}{\hat{s}}
\nc{\bfb}{{\bf b}}
\nc{\cL}{{\mc L}}

\nc{\ds}{\diamondsuit}
\nc{\bB}{{\bf B}}
\nc{\pqmn}{{p+m|q+n}}
\nc{\mn}{{m|n}}
\nc{\gls}{\gl_\ds}
\nc{\gld}{\gl_d}
\nc{\glmn}{\gl_{m|n}}
\nc{\glpm}{\gl_{p+m}}
\nc{\glsm}{t^{-1} \gls [t^{-1}]}
\nc{\gldm}{t^{-1} \gld [t^{-1}]}
\nc{\fhm}{t^{-1} \fh [t^{-1}]}
\nc{\fnpm}{t^{-1} \np [t^{-1}]}
\nc{\fnmm}{t^{-1} \nm [t^{-1}]}

\nc{\Vac}{V_{{\rm crit}}}
\nc{\fz}{{\mf z}}
\nc{\fzr}{\fz(\wh{\gl}_\ds)}
\nc{\fzd}{\fz(\wh{\gl}_d)}
\nc{\fgm}{t^{-1} \fg [t^{-1}]}
\nc{\fgdm}{t^{-1} \gld [t^{-1}]}
\nc{\fgrm}{t^{-1} \gls [t^{-1}]}
\nc{\fzg}{\fz(\wh{\fg})}
\nc{\fzmn}{\fz(\wh{\gl}_{m|n})}

\nc{\ev}{{\rm ev}}
\nc{\evz}{\un{\rm ev}_\z}
\nc{\evzga}{{\rm ev}^\ga_\z}
\nc{\evjz}{\un{\ev}_{(\z, \jjp)}}

\nc{\hc}{\mf f}
\nc{\qc}{S_0}
\nc{\st}{{st}}
\nc{\T}{{\rm T}}
\nc{\cT}{\mc{T}}
\nc{\Xl}{{\bd X}_{\! \ell}}
\nc{\End}{{\rm End}}

\nc{\xai}{x^a_i}
\nc{\xaj}{x^a_j}
\nc{\xbi}{x^b_i}
\nc{\xbj}{x^b_j}
\nc{\yar}{y^a_r}
\nc{\yas}{y^a_s}
\nc{\ybr}{y^b_r}
\nc{\ybs}{y^b_s}
\nc{\yai}{y^a_i}
\nc{\ybi}{y^b_i}
\nc{\pyai}{\pa_{\yai}}
\nc{\pybi}{\pa_{\ybi}}
\nc{\pxai}{\pa_{\xai}}
\nc{\pxaj}{\pa_{\xaj}}
\nc{\pxbi}{\pa_{\xbi}}
\nc{\pxbj}{\pa_{\xbj}}
\nc{\pyar}{\pa_{\yar}}
\nc{\pyas}{\pa_{\yas}}
\nc{\pybr}{\pa_{\ybr}}
\nc{\pybs}{\pa_{\ybs}}
\nc{\pxaip}{\pa_{x^a_{i'}}}

\nc{\blb}{(\!(}
\nc{\brb}{)\!)}
\nc{\lb}{\left(}
\nc{\rb}{\right)}
\nc{\lbb}{\left(\!}
\nc{\rbb}{\! \right)}
\nc{\ls}{\left[}
\nc{\rs}{\! \right]}
\nc{\llangle}{\big\langle}
\nc{\rrangle}{\big\rangle}
\nc{\np}{\mf{n}_+}
\nc{\nm}{\mf{n}_-}

\nc{\twa}{t_{w_a}}
\nc{\twb}{t_{w_b}}
\nc{\twi}{t_{w_i}}
\nc{\tzi}{t_{z_i}}
\nc{\tzj}{t_{z_j}}
\nc{\otwa}{\ov{t}_{w_a}}
\nc{\otwb}{\ov{t}_{w_b}}
\nc{\otzi}{\ov{t}_{z_i}}
\nc{\otzj}{\ov{t}_{z_j}}

\nc{\jp}{\ga}
\nc{\jjp}{{\bd \ga}}
\nc{\bjp}{\xi}
\nc{\bjjp}{{\bd \xi}}
\nc{\va}{d}
\nc{\rva}{l}
\nc{\J}{{\rm J}}
\nc{\Lr}{\cL_{\ds}}

\nc{\dd}{{d^\prime}}
\nc{\el}{{\bd \ell}}

\nc{\gz}{\fg({\z, \jjp})}
\nc{\gldz}{\gld({\z, \jjp})}
\nc{\gldzb}{\gld({\z, \jjp})}
\nc{\glsz}{\gls({\z, \jjp})}
\nc{\glsw}{\gls({\w,  \bjjp})}
\nc{\glssw}{\gl_\pqmn({\w,  \bjjp})}
\nc{\glpmwa}{\gl_{p+m}({w_a,  \bjp_a})}

\nc{\vpd}{\phi}
\nc{\vps}{\varphi}

\nc{\Adzw}{\cA_d^{\w, \, \bjjp}(\z,  \jjp)}
\nc{\Arwz}{\cA_{\ds}^{\z,  \jjp}(\w, \bjjp)}
\nc{\Aswz}{\cA_{\pqmn}^{\z,  \jjp}(\w, \bjjp)}
\nc{\Ldzw}{\cL_d^{\w, \, \bjjp}(\z,  \jjp)}
\nc{\wLdzw}{\wh{\cL}_d^{\w, \, \bjjp}(\z,  \jjp)}
\nc{\Lrwz}{\cL_{\ds}^{\z,  \jjp}(\w, \bjjp)}
\nc{\wLrwz}{\wh{\cL}_{\ds}^{\z,  \jjp}(\w, \bjjp)}

\nc{\Cdzw}{\ov\cA_d^{\w, \, \bjjp}(\z,  \jjp)}
\nc{\Crwz}{\ov\cA_{\ds}^{ \z,  \jjp}(\w, \bjjp)}
\nc{\Cswz}{\ov\cA_{\pqmn}^{ \z,  \jjp}(\w, \bjjp)}
\nc{\rLs}{\ov\cL_{\ds}}
\nc{\rLdzw}{\ov\cL_d^{\w, \, \bjjp}(\z,  \jjp)}
\nc{\rLswz}{\ov\cL_{\ds}^{\z,  \jjp}(\w, \bjjp)}
\nc{\Ldmu}{\cL_d^{\mu_{\bjjp}^{\w}}(\z, \jjp)}
\nc{\wLdmu}{\wh{\cL}_d^{\mu_{\bjjp}^{\w}}(\z, \jjp)}

\nc{\Pmuzz}{\Psi^{\mu, \jjp, z}_{\z}}
\nc{\Pmuzw}{\Psi^{\mu, \jjp, w}_{\z}}
\nc{\Phmuz}{\ov\Psi^{\mu}_{(\z, \jjp), \, z}}
\nc{\Phmuw}{\ov\Psi^{\mu}_{(\z, \jjp),\, w}}
\nc{\Pmuz}{\Psi^{\mu}_{(\z, \jjp)}}
\nc{\Phz}{\ov\Psi_{(\z, \jjp), \, z}}
\nc{\Phw}{\ov\Psi_{(\z, \jjp),\, w}}

\nc{\Xd}{{\bd X}_{\! \dd}}
\nc{\Xpqmn}{{\bd X}_{\! \el}}
\nc{\gldzi}{\gld({z_i, \jp_i})}
\nc{\glszi}{\gls({z_i, \jp_i})}
\nc{\glswa}{\gls({w_a,  \bjp_a})}
\nc{\hswa}{\fh_\ds({w_a,  \bjp_a})}
\nc{\bswa}{\fb_\ds({w_a,  \bjp_a})}

\nc{\gltga}{\gl^{(\ga)}_\ds[t]}
\nc{\htga}{\fh^{(\ga)}_\ds[t]}
\nc{\btga}{\fb^{(\ga)}_\ds[t]}
\nc{\gltd}{\gl^{(d)}_\ds[t]}

\nc{\Vk}{\bn V{(k_1, \ldots, k_{\dd})}}
\nc{\Wmu}{\bn W{(\mu_1, \ldots, \mu_\el)}}
\nc{\Wmub}{{\bn W{(\ov{\mu}_1, \ldots, \ov{\mu}_\el)}}}
\nc{\Wmubk}{{\bn W{(\ov{\mu}_1, \ldots, \ov{\mu}_\el)}}_{[k_1, \ldots, k_{\dd}]}}

\nc{\zpz}{\blb z^{-1}, \pz^{-1} \brb}
\nc{\dds}{{\bf d}_{\F}}
\nc{\ppxai}{p_{x^a_i}}
\nc{\ppyar}{p_{y^a_r}}
\nc{\ppxbj}{p_{x^b_j}}
\nc{\ppybs}{p_{y^b_s}}
\nc{\oD}{\ov{\cD}}
\nc{\ovpd}{\ov\vpd}
\nc{\ovps}{\ov\vps}

\nc{\disp}{\displaystyle}
\nc{\tns}{\kern-1.1pt}

\title[Dualities of Gaudin models with irregular singularities]
{Dualities of Gaudin models with irregular singularities for general linear Lie (super)algebras}

\author[W. K. Cheong]{Wan Keng Cheong}
\address{Department of Mathematics, National Cheng Kung University, Tainan, Taiwan 701401}
\email{keng@ncku.edu.tw}

\author[N. Lam]{Ngau Lam}
\address{Department of Mathematics, National Cheng Kung University, Tainan, Taiwan 701401}
\email{nlam@ncku.edu.tw}

\begin{abstract}

We prove an equivalence between the actions of the Gaudin algebras with irregular singularities for $\gld$ and $\gl_\pqmn$ on the Fock space of $d(p+m)$ bosonic and $d(q+n)$ fermionic oscillators. This establishes a duality of $(\gld, \gl_\pqmn)$ for Gaudin models. As an application, we show that the Gaudin algebra with irregular singularities for $\gl_\pqmn$ acts cyclically on each weight space of a certain class of infinite-dimensional modules over a direct sum of Takiff superalgebras over $\gl_\pqmn$ and that the action is diagonalizable with a simple spectrum under a generic condition.
We also study the classical versions of Gaudin algebras with irregular singularities and demonstrate a duality of $(\gld, \gl_\pqmn)$ for classical Gaudin models.

\end{abstract}

\maketitle

\setcounter{tocdepth}{1}

\section{Introduction}

Rybnikov \cite{Ry06} and Feigin, Frenkel, and Toledano Laredo \cite{FFTL} introduce Gaudin models with irregular singularities associated to any simple Lie algebra over $\C$ based on the works \cite{FF, FFR, Fr05}. Their constructions can also be applied to any general linear Lie (super)algebra.

A number of dualities relating the Gaudin models for a pair of general linear Lie (super)algebras are demonstrated under some restrictions on singularities; see \cite{ChL25-1, HM, MTV09, TU}.
Remarkably, Vicedo and Young \cite{VY} establish a duality of Gaudin models for any pair of general linear Lie algebras without imposing such restrictions.

The main goal of this paper is to generalize Vicedo--Young's duality to the super setting.
We will prove a duality of Gaudin models with irregular singularities for the pair $(\gld, \gl_\pqmn)$, where $\gld$ and $\gl_\pqmn$ are the general linear Lie (super)algebras defined in \secref{gl}.

Now we describe our duality in more concrete terms.
Let $\w=(w_1, \ldots, w_\dd)$ and $\z=(z_1, \ldots, z_{\el})$ be sequences of complex numbers,
and let $\bjjp=(\bjp_1,\ldots, \bjp_{\dd})$ and $\jjp=(\jp_1,\ldots, \jp_{\el})$ be sequences of positive integers, where $\dd$ and $\el$ are some positive integers satisfying $\dd \le d$ and $\el \le p+q+m+n$; see \secref{Fock} for details.
Let $t$ be an even variable and write $t_a=t-a$ for $a \in \C$.
We consider the direct sums of Lie (super)algebras
$$
\gldzb:=\bigoplus_{i=1}^{\el} \gld[\tzi]/\tzi^{\jp_i}\gld[\tzi]
$$
and
$$
\glssw:=\bigoplus_{i=1}^{\dd} \gl_\pqmn[\twi]/\twi^{\bjp_i}\gl_\pqmn[\twi].
$$
Here, for instance, $\gld[\tzi]:=\gld \otimes \C[\tzi]$ and $\tzi^{\jp_i}\gld[\tzi]:=\gld \otimes \tzi^{\jp_i} \C[\tzi]$. The quotient Lie (super)algebras $\gld[\tzi]/\tzi^{\jp_i}\gld[\tzi]$ and $\gl_\pqmn[\twi]/\twi^{\bjp_i}\gl_\pqmn[\twi]$ are called Takiff (super)algebras over $\gld$ and $\gl_\pqmn$, respectively; see \secref{takiff}.

\sloppy
Let $\Adzw$ denote the \emph{Gaudin algebra for $\gld$ with singularities of orders $\jp_i$ at $z_i$, $i=1, \ldots, \el$, relative to $\w$ and $\bjjp$}, 
and let $\Aswz$ denote the \emph{Gaudin algebra for $\gl_\pqmn$ with singularities of orders $\bjp_i$ at $w_i$, $i=1, \ldots, \dd$, relative to $\z$ and $\jjp$}; see \secref{Gaudin-gld}, \secref{Gaudin-gls} and \eqref{Gaudin-Jordan}.
The algebras $\Adzw$ and $\Aswz$ are commutative subalgebras of the universal enveloping algebras $U \big(\gldzb \big)$ and $U \big(\glssw \big)$, respectively.
Also, for $1 \le i \le \el$, the singularity at $z_i$ is said to be regular if $\jjp_i=1$ and irregular otherwise. The singularity type (either regular or irregular) of each $w_j$ is defined similarly.

Let $\F$ be the polynomial superalgebra generated by $\xai$ and $\yar$, for $i=1, \ldots, m+n$, $r=1, \ldots, p+q$ and $a=1, \ldots, d$. Here the variable $\xai$ (resp., $\yar$) is even for $1 \le i \le m$ (resp., $1 \le r \le p$) and is odd otherwise.
There are commuting actions of $\gld$ and $\gl_\pqmn$ on $\F$ that form a Howe dual pair \cite{CLZ}.
The superalgebra $\F$ can be realized as the Fock space of $d(p+m)$ bosonic and $d(q+n)$ fermionic oscillators (cf. \cite{CL03, LZ06}).
Let $\cD$ be the corresponding Weyl superalgebra, which is the superalgebra generated by $\xai$ and $\yar$ as well as their derivatives $\frac{\pa}{\pa x^a_i}$ and $\frac{\pa}{\pa y^a_r}$.
The superalgebra $\cD$ naturally acts on $\F$ and is a subalgebra of the endomorphism algebra $\End(\F)$ of $\F$.
We construct superalgebra homomorphisms $\vpd : U \big(\gldzb \big) \longrightarrow \cD$ and $\vps : U \big(\glssw \big) \longrightarrow \cD$ (see \propref{gld-map} and \propref{glr-map}) and obtain the \emph{duality of $(\gld, \gl_\pqmn)$ for Gaudin models with irregular singularities}.

\begin{thm} [\thmref{Gaudin-duality}] \label{main1}
$
\vpd \big( \Adzw \big)=\vps \big(\Aswz \big).
$
\end{thm}

The maps $\vpd$ and $\vps$ induce an action of $\Adzw$ and an action of $\Aswz$ on the Fock space $\F$, respectively.
\thmref{main1} also says that these actions are equivalent.

A few remarks are in order. In \cite{ChL25-1} (and also in \cite{HM, MTV09, TU} if some of $p,q,m,n$ are set to 0), a duality of $(\gld, \gl_\pqmn)$ in the same spirit is established when the entries of $\jjp$ and $\bjjp$ are all 1 (i.e., the singularities at $z_i$'s and $w_j$'s are all regular).
These conditions mean that the duality considered there concerns only the Gaudin models with the simplest possible singularities.
\thmref{main1} is, however, about a duality of Gaudin models with general singularity types (i.e., no restrictions are imposed on the singularities at $z_i$'s and $w_j$'s).
After specializing to $p=q=n=0$, the theorem recovers the duality of $(\gld, \gl_m)$ obtained by Vicedo and Young \cite[Theorem 4.8]{VY} in the bosonic setting. 
Furthermore, \thmref{main1} yields a duality in the fermionic setting by taking $p=q=m=0$; see \corref{fermi}.
Specializing to $q=n=0$, the theorem also deduces the duality of $(\gld, \gl_{m+p})$ for the bosonic oscillators in which the Fock space $\F$ decomposes into a direct sum of tensor products of infinite-dimensional modules over the general linear Lie algebra $\gl_{m+p}$; see \corref{bosonic}.

We can apply \thmref{main1} to study the Gaudin algebras in the case of $\jjp=(1^\el)$, where $\el=p+q+m+n$.
Using the properties of the action of $\cA_d^{\w, \bjjp}(\z, (1^\el))$ on finite-dimensional $\gld$-modules, we can obtain some consequences for the action of $\cA_\pqmn^{\z,  (1^{\el}) }(\w, \bjjp)$ on a certain class of $\glssw$-modules.
Let $\bn M=M_1 \otimes \cdots \otimes M_{\dd}$, where $M_i$ is a certain infinite-dimensional $\gl_\pqmn[\twi]/\twi^{\bjp_i}\gl_\pqmn[\twi]$-module for $1 \le i \le \dd$.
We will prove in \thmref{Gaudin-app} that every weight space of $\bn M$ is a cyclic $\cA_\pqmn^{\z,  (1^{\el}) }(\w, \bjjp)$-module, and that $\cA_\pqmn^{\z,  (1^{\el}) }(\w, \bjjp)$ is diagonalizable with a simple spectrum on the weight space for generic $\w$ and $\z$.

In this paper, we also study the classical versions of Gaudin models.
Let $\cS \big(\gldzb \big)$ and $\cS \big(\glssw \big)$ be the supersymmetric algebras of $\gldz$ and $\glssw$, respectively.
They are also Poisson superalgebras and may be viewed as classical counterparts of $U \big(\gldzb \big)$ and $U \big(\glssw \big)$, respectively.
Let $\Cdzw \subseteq \cS \big(\gldzb \big)$ be the \emph{classical Gaudin algebra for $\gld$ with singularities of orders $\jp_i$ at $z_i$, $i=1, \ldots, \el$, relative to $\w$ and $\bjjp$}, 
and let $\Cswz \subseteq \cS \big(\glssw \big)$ be the \emph{classical Gaudin algebra for $\gl_\pqmn$ with singularities of orders $\bjp_i$ at $w_i$, $i=1, \ldots, \dd$, relative to $\z$ and $\jjp$}; see \secref{cl-gld} and \secref{cl-glr}.
The classical Gaudin algebras $\Cdzw$ and $\Cswz$ are Poisson-commutative.

Let $\oD$ be the polynomial superalgebra generated by the variables $\xai$, $\yar$, $\ppxai$ and $\ppyar$, for $i=1, \ldots, m+n$, $r=1, \ldots, p+q$ and $a=1, \ldots, d$, where $\xai$ and $\ppxai$ (resp., $\yar$ and $\ppyar$) are even for $1 \le i \le m$ (resp., $1 \le r \le p$) and are odd otherwise. 
It is a Poisson superalgebra with the Poisson bracket defined by \eqref{Pb1} and \eqref{Pb2}.
There also exist Poisson superalgebra homomorphisms $\ovpd : \cS \big(\gldzb \big) \longrightarrow \oD$ and $\ovps : \cS \big(\glssw \big) \longrightarrow \oD$; see \propref{map-classical}.
The duality in \thmref{main1} has a classical version, which we call the \emph{duality of $(\gld, \gl_\pqmn)$ for classical Gaudin models with irregular singularities}.

\begin{thm} [\thmref{class-dual}] \label{main2}
$\ovpd \big( \Cdzw \big)=\ovps \big(\Cswz \big)$.
\end{thm}

\thmref{main2} recovers the duality of $(\gld, \gl_m)$ for classical Gaudin models due to Vicedo and Young \cite[Theorem 3.2]{VY} by taking $p=q=n=0$ and the fermionic counterpart \cite[Theorem 3.4]{VY} by taking $p=q=m=0$.

We organize the paper as follows.
In \secref{Pre}, we review the background materials needed in this paper.
In \secref{GM}, we discuss Gaudin models with irregular singularities and their fundamental properties.
In \secref{duality}, we introduce a joint action of the Gaudin algebras $\Adzw$ and $\Aswz$ on the Fock space $\F$.
We prove \thmref{main1}, establishing a duality of $(\gld, \gl_\pqmn)$ for Gaudin models. Specializing to $\jjp=(1^\el)$, we give an application to a class of modules over $\glssw$ and obtain \thmref{Gaudin-app}.
In \secref{classical}, we introduce the classical Gaudin algebras $\Cdzw$ and $\Cswz$. We prove \thmref{main2} and obtain a duality of $(\gld, \gl_\pqmn)$ for classical Gaudin models.

\vskip 0.3cm
\noindent{\bf Notations.}
Throughout the paper, the symbol $\Z$ (resp., $\N$ and $\Zp$) stands for the set of all (resp., positive and non-negative) integers, the symbol $\C$ for the field of complex numbers, and the symbol $\Z_2:=\{\bz, \bo\}$ for the field of integers modulo 2. All vector spaces, algebras, tensor products, etc., are over $\C$.
{\bf We fix $d \in \N$ and $p,q,m, n \in \Zp$.}

\bigskip

\section{Preliminaries} \label{Pre}

In this section, we review general linear Lie (super)algebras, column determinants, Berezinians, and pseudo-differential operators.

\subsection{The general linear Lie (super)algebra} \label{gl}

For $p, q, m, n\in \Zp$ that are not all zero, let
$$
\I=\{i \in \N \, | \, 1 \le i \le p+q+m+n\}.
$$
For $i \in \I$, define $|i| \in \Z_2$ by
$$
|i|=\begin{cases}
\, \bz & \  \  \mbox{if} \ \  i \in \{1, \ldots, p \} \cup \{ p+q+1, \ldots, p+q+m\};\\
\, \bo & \  \  \mbox{otherwise}.
\end{cases}
$$
Let $\{e_i \, | \, i \in \I\}$ be a basis for the superspace $\C^{p|q}\oplus\C^{m|n}$ such that $\{e_i \, | \, 1\le i \le p+q\}$ and
$\{e_{p+q+i} \, | \, 1\le i \le m+n\}$ are respectively the standard homogeneous bases for $\C^{p|q}$ and $\C^{m|n}$. In other words, the parity of $e_i$ is given by
$|e_i|=|i|$ for $i \in \I$.

For any $i, j \in \I$, let $E^i_j$ denote the $\C$-linear endomorphism on $\C^{p|q}\oplus\C^{m|n}$ defined by
$$
E^i_j (e_k)=\delta_{j, k} e_i \quad \mbox{ for $k \in \I$,}
$$
where $\delta$ is the Kronecker delta. The parity of $E^i_j$ is given by $|E^i_j|=|i|+|j|$.
The superspace of $\C$-linear endomorphisms on $\C^{p|q}\oplus\C^{m|n}$ is a Lie superalgebra, called a \emph{general linear Lie (super)algebra} and denoted by $\gl_\pqmn$, with commutation relations given by
$$
[E^i_j, E^k_l]=\de_{j,k} E^i_l-(-1)^{(|i|+|j|)(|k|+|l|)}\de_{i, l} E^k_j \qquad \mbox{for $i, j, k, l \in \I$.}
$$
Note that $\{ E^i_j \, | \, i, j \in \I \}$ is a homogeneous basis for $\gl_\pqmn$.

For the rest of the paper, we use the symbol $\ds:=\pqmn$.
Write
$$
\gls=\gl_\pqmn.
$$
Let $\mf{b}_{\ds}=\bigoplus_{\substack{ i,j \in \I, i \le j} }  \C E^i_j$ be a Borel subalgebra of $\gls$.
The corresponding Cartan subalgebra $\fh_{\ds}$ has a basis $\{ E^i_i \, | \, i \in \I \}$, and the dual basis in $\fh_{\ds}^{*}$ is denoted by $\{ \ep_i \, | \, i \in \I \}$, where the parity of $\ep_i$ is given by $|\ep_i|=|i|$.
Note that the Borel subalgebra $\mf{b}_{\ds}$ is not the standard one unless some of $p, q, m, n$ are set to 0.

For $m=d$ and $p=q=n=0$,
we write
$$
\gld=\gl_{d|0}
$$
and
$$
\hspace{1cm}  e_{ij}=E^i_j \qquad \mbox{for $i, j=1, \ldots, d$.}
$$

\subsection{Column determinants and Berezinians} \label{Ber}

Let $\cA$ be an associative unital superalgebra over $\C$.
The parity of a homogeneous element $a \in \cA$ is denoted by $|a|$, which lies in $\Z_2$.
The superalgebra $\cA$ is naturally a Lie superalgebra with supercommutator
\begin{equation} \label{s-comm}
[a, b]:=ab - (-1)^{|a||b|} ba
\end{equation}
for homogeneous elements $a, b \in \cA$.

Fix $k \in \N$. For any $k \times k$ matrix $A=\big[a_{i,j}\big]_{i,j=1,\ldots,k}$ over $\cA$, the \emph{column determinant} of $A$ is defined to be
$$
\cdet (A)=\sum_{\si\in\mf{S}_k} (-1)^{l(\si)} \,  a_{\si(1),1}\ldots a_{\si(k),k}.
$$
Here $\fS_k$ denotes the symmetric group on $\{1, \ldots, k\}$, and $l(\si)$ denotes the length of $\si$.
The \emph{row determinant} of $A$ is defined to be
$$
\rdet (A)=\sum_{\si\in\mf{S}_k} (-1)^{l(\si)} \,  a_{1, \si(1)}\ldots a_{k, \si(k)}.
$$
Evidently, $\rdet(A)=\cdet(A^t)$, where $A^t$ is the transpose of $A$.
If all the entries of $A$ commute, then the column determinant and the row determinant of $A$ coincide, and we define the \emph{determinant} of $A$ to be 
$$
\det(A)=\cdet(A)=\rdet(A).
$$

Assume that $A$ has a two-sided inverse $A^{-1}=\big[\wt{a}_{i,j}\big]$.
For $i, j=1, \ldots, k$, the \emph{$(i,j)$th quasideterminant} of $A$ is defined to be $|A|_{ij}:=\wt{a}_{j,i}^{-1}$ provided that $\wt{a}_{j,i}$ has an inverse in $\cA$.
If $k \ge 2$, then
$$
|A|_{ij}=a_{i,j}-r_i^j (A^{ij})^{-1} c_j^i,
$$
where $A^{ij}$ is the $(k-1) \times (k-1)$ submatrix of $A$ obtained by deleting the $i$th row and the $j$th column of $A$ and is assumed to be invertible, $r_i^j$ is the $i$th row of $A$ with $a_{i,j}$ removed, and $c_j^i$ is the $j$th column of $A$ with $a_{i,j}$ removed (see \cite[Proposition 1.2.6]{GGRW}).
Following \cite{GGRW}, it is convenient to write
$$
|A|_{ij}=
\begin{vmatrix}
 a_{1,1} 	& \ldots & a_{1,j} &\ldots & a_{1,k}\\
 \ldots	& \ldots & \ldots  &\ldots & \ldots	\\
 a_{i,1}& \ldots & \fbox{$a_{i,j}$} &\ldots &a_{i,k}\\
 \ldots	& \ldots & \ldots  &\ldots & \ldots	\\
 a_{k,1}&\ldots &a_{k,j} &\ldots &a_{k,k}
\end{vmatrix}.
$$
For $i=1,\ldots,k$, we define
$$
d_i(A)=\begin{vmatrix}
	 a_{1,1} 	& \ldots & a_{1,i}\\
	 \ldots	& \ldots & \ldots \\
	 a_{i,1}& \ldots & \fbox{$a_{i,i}$}
\end{vmatrix},\
$$
called the \emph{principal quasiminors} of $A$.

Let $A=\big[a_{i,j}\big]_{i,j=1,\ldots,k}$ be a $k \times k$ matrix over $\cA$.
For any nonempty subset $P=\{i_1<\ldots<i_\ell \}$ of $\{1,\ldots,k\}$, the matrix $A_{P}:=\big[a_{i,j}\big]_{i,j \in P}$ is called a \emph{standard submatrix} of $A$.
Moreover, we say that $A$ is \emph{sufficiently invertible} if every principal quasiminor of $A$ is well defined, and that $A$ is \emph{amply invertible} if each of its standard submatrices is sufficiently invertible.

Let $(s_1,\ldots,s_{m+n})$ be a sequence of 0's and 1's such that exactly $m$ of the $s_i$'s are 0 and the others are 1.
We call such a sequence a \emph{$0^m1^n$-sequence}.
Every $0^m1^n$-sequence can be written in the form
$(0^{m_1}, 1^{n_1}, \ldots, 0^{m_r}, 1^{n_r}),$
where the sequence begins with $m_1$ copies of $0$'s, followed by $n_1$ copies of $1$'s, and so on.
The set of all $0^m1^n$-sequences is denoted by $\cS_{\mn}$.

Let $\s=(s_1,\ldots,s_{m+n}) \in \cS_{\mn}$.
For any $\si \in \fS_{m+n}$ and any $(m+n)\times (m+n)$ matrix $A:=[a_{i,j}]$ over $\cA$, we define $\s^\si=\left(s_{\si^{-1}(1)},s_{\si^{-1}(2)},\ldots,s_{\si^{-1}(m+n)} \rb$ and $A^\si=\big[a_{\si^{-1}(i),\si^{-1}(j)} \big]$.
We say that $A$ is of type $\s$ if $a_{i,j}$ is a homogeneous element of parity $|a_{i,j}|=\bs_i+\bs_j$ for $1\le i, j \le m+n$.

Following the definition of \cite{HM}, for any $(m+n)\times (m+n)$ sufficiently invertible matrix $A$ of type $\s$ over $\cA$, the \emph{Berezinian of type $\s$} of $A$ is defined to be
$$
\Bers (A)=d_1(A)^{\hs_1}\ldots d_{m+n}(A)^{\hs_{m+n}},
$$
where $\hs_i:=(-1)^{s_i}$ for $1 \le i \le m+n$.
The original formulation of Berezinians, in the case where $\cA$ is supercommutative, was given in \cite{Ber}.

\begin{prop} [{cf. \cite[Proposition 3.5]{HM}}] \label{decomp}
Let $A$ be an $(m+n)\times (m+n)$ amply invertible matrix of type $\s$ over $\cA$.
Fix $k\in \{1,\ldots,m+n-1\}$. We write
\begin{equation} \label{block}
A=\begin{bmatrix} W & X\\ Y&Z \end{bmatrix},
\end{equation}
where $W,X,Y,Z$ are respectively $k\times k$, $k\times(m+n-k)$, $(m+n-k)\times k$, and $(m+n-k)\times (m+n-k)$ matrices.
Then $W$ and $Z-YW^{-1}X$ are sufficiently invertible matrices of types $\s^\prime:=(s_1,\ldots,s_k)$ and $\s^\dpr:=(s_{k+1},\ldots,s_{m+n})$, respectively. Moreover,
$$
\Bers (A)={\rm Ber}^{\s^\prime} (W) \cdot {\rm Ber}^{\s^\dpr} \big(Z-YW^{-1}X \big).
$$
\end{prop}

An $(m+n)\times (m+n)$ matrix $A=\big[a_{i,j}\big]$ over $\cA$ is called a \emph{Manin matrix} of type $\s$ if $A$ is a matrix of type $\s$ satisfying
$$
[a_{i,j}, a_{k,l}]=(-1)^{s_i s_j+s_i s_k+s_j s_k}[a_{k,j}, a_{i, l}]
$$
for $1 \le i,j,k,l \le m+n$.
Evidently, if $\cA$ is supercommutative, then any $(m+n)\times (m+n)$ matrix of type $\s$ is automatically a Manin matrix.

Let us recall some useful facts about Manin matrices.

\begin{prop}[{cf. \cite[Section 3]{HM}}] \label{basics}

Let $A$ be an $(m+n)\times (m+n)$ Manin matrix of type $\s$ over $\cA$. Then

\begin{enumerate}[\normalfont(i)]

\item If $P =\{i_1<\ldots<i_\ell \}$ is a nonempty subset of $\{1,\ldots,m+n\}$, then the standard submatrix $A_{P}$ of $A$ is a Manin matrix of type ${\s}_P:=(s_{i_1},\ldots s_{i_\ell})$.

\item For any $\si \in \fS_{m+n}$, $A^\si$ is a Manin matrix of type $\s^\si$.

\end{enumerate}

\end{prop}

\begin{prop} [{\cite[Proposition 3.6]{HM}}] \label{HM}

Let $A$ be an $(m+n)\times (m+n)$ amply invertible Manin matrix of type $\s$ over $\cA$.
Then ${\rm Ber}^{\s^\si} (A^\si)=\Bers (A)$ for any $\si \in \fS_{m+n}$.
\end{prop}

\begin{prop} [{\cite[Proposition 4.4]{ChL25-2}}]  \label{decomp-2}
Let $A$ be an $(m+n)\times (m+n)$ amply invertible Manin matrix of type $\s$ over $\cA$. Write $A$ as in \eqref{block}.
Then $W-XZ^{-1} Y$ and $Z$ are sufficiently invertible matrices of types $\s^\prime:=(s_1,\ldots,s_k)$ and $\s^\dpr:=(s_{k+1},\ldots,s_{m+n})$, respectively.
Moreover,
$$
\Bers (A)={\rm Ber}^{\s^\dpr} (Z)  \cdot {\rm Ber}^{\s^\prime} \big(W-XZ^{-1} Y \big).
$$
\end{prop}

\begin{prop}[{\cite[Lemma 8]{CFR}}] \label{CFR}
Let $\s_0=(0^m)$, and let $A$ be an $m\times m$ sufficiently invertible Manin matrix of type $\s_0$ over $\cA$. Then $\Ber^{\s_0} (A)=\cdet  (A)$.
\end{prop}

\begin{prop} \label{Ber1n}
Suppose that $\cA$ is supercommutative.
Let $\s_1=(1^n)$, and let $A$ be an $n \times n$ sufficiently invertible matrix of type $\s_1$ over $\cA$. Then $\Ber^{\s_1} (A)=[\det  (A)]^{-1}$.
\end{prop}

\begin{proof}
By hypothesis, $A$ is a matrix of type $(0^n)$ with commuting entries.
By \propref{CFR}, $\det  (A)=d_1(A) \ldots d_n(A)$. It follows that
$[\det  (A)]^{-1}= d_1(A)^{-1}  \ldots d_n(A)^{-1}=\Ber^{\s_1} (A)$
as the elements $d_i(A)$ commute.
\end{proof}

\subsection{Pseudo-differential operators} \label{diff-op}

Let $\cA$ be an associative unital superalgebra over $\C$, and let $z$ be an even variable.
We denote by $\cA [\tns[z]\tns]$ (resp., $\cA \blb z \brb$) the superalgebra of formal power series (resp., formal Laurent series) in $z$ with coefficients in $\cA$.

The superalgebra $\cA \blb z^{-1} \brb$ is equipped with an even derivation $\pz$ defined by
$$
  \pz \lb \sum_{i=-\infty}^r a_i z^i \rb=\sum_{i=-\infty}^r i a_i z^{i-1}, \qquad  r \in \Z \quad \text{and} \quad a_i \in \cA.
$$
Let $\cA \blb z^{-1} \brb \blb \pz^{-1} \brb$ be the superspace of formal Laurent series in $\pz^{-1}$ with coefficients in $\cA \blb z^{-1} \brb$.
The multiplication on $\cA \blb z^{-1} \brb$ extends to a multiplication on $\cA \blb z^{-1} \brb \blb \pz^{-1} \brb$ by the rules
\begin{equation}\label{rule}
\pz \pz^{-1}=\pz^{-1} \pz=1, \quad \pz^{i} \phi=\sum_{k=0}^{\infty}  \binom{i}{k}  \,  \pz^{k}(\phi) \, \pz^{i-k},  \quad \mbox{for $\phi \in \cA \blb z^{-1} \brb$ and $i \in \Z$.}
\end{equation}
Here $\disp{\binom{i}{k}:=\frac{i(i-1)\ldots (i-k+1)}{k!}}$ for $i \in \Z$ and $k \in \Zp$.

\begin{lem}
$\cA \blb z^{-1} \brb \blb \pz^{-1} \brb$ is an associative unital superalgebra over $\C$.
\end{lem}

\begin{proof}
We need only show that the associative law for multiplication holds in $\cA \blb z^{-1} \brb \blb \pz^{-1} \brb$.
To establish the law, it suffices to show that
\begin{equation}\label{ass}
z^{i_1} \pz^{j_1} \big( z^{i_2} \pz^{j_2}  z^{i_3} \pz^{j_3} \big) = \big(z^{i_1} \pz^{j_1} z^{i_2} \pz^{j_2} \big) z^{i_3} \pz^{j_3}
\end{equation}
for $i_r, j_r \in \Z$, $r=1,2,3$. By \eqref{rule},
$$
z^{i_1} \pz^{j_1} \big( z^{i_2} \pz^{j_2}  z^{i_3} \pz^{j_3} \big) = \sum_{k=0}^\infty a_k z^{i_1+i_2+i_3-k} \pz^{j_1+j_2+j_3-k},
$$
where 
$$
a_k=\sum_{r=0}^k \binom{j_1}{r}  \binom{i_2+i_3-k+r}{r}  \binom{j_2}{k-r} \binom{i_3}{k-r} r! (k-r)!,
$$
and
$$
 \big(z^{i_1} \pz^{j_1} z^{i_2} \pz^{j_2} \big) z^{i_3} \pz^{j_3} = \sum_{k=0}^\infty b_k z^{i_1+i_2+i_3-k} \pz^{j_1+j_2+j_3-k},
$$
where
$$
b_k=\sum_{r=0}^k \binom{j_1}{r}  \binom{i_2}{r} \binom{j_1+j_2-r}{k-r} \binom{i_3}{k-r}  r! (k-r)!.
$$
Using Vandermonde's identity:
$$
\binom{i+j}{r}=\sum_{s=0}^r \binom{i}{s}  \binom{j}{r-s},  \qquad  i, j \in \Z \,\, \text{ and } \,\,  r \in \Zp,
$$
we have for $k \in \Zp$,
\begin{eqnarray*}
a_k
&=&\sum_{r=0}^k \binom{j_1}{r}   \binom{j_2}{k-r} \lb \sum_{s=0}^r  \binom{i_2}{r-s}  \binom{i_3-k+r}{s} \rb \binom{i_3}{k-r}  r! (k-r)!\\
&=&\sum_{r=0}^k \sum_{s=0}^r \binom{j_1}{r}   \binom{j_2}{k-r}  \binom{i_2}{r-s} \binom{i_3}{k-r+s}  \frac{r! (k-r+s)!}{s!}
\end{eqnarray*}
and
\begin{eqnarray*}
b_k
&=&\sum_{r=0}^k   \binom{i_2}{r}  \binom{i_3}{k-r} \binom{j_1}{r} \lb \sum_{s=0}^{k-r} \binom{j_1-r}{s} \binom{j_2}{k-r-s} \rb r! (k-r)! \\
&=&\sum_{r=0}^k \sum_{s=0}^{k-r}  \binom{i_2}{r}  \binom{i_3}{k-r} \binom{j_1}{r+s}  \binom{j_2}{k-r-s}  \frac{(r+s)! (k-r)!}{s!}  \\
&=&\sum_{\ell=0}^k \sum_{s=0}^{\ell}  \binom{i_2}{\ell-s}  \binom{i_3}{k-\ell+s} \binom{j_1}{\ell}  \binom{j_2}{k-\ell}  \frac{ \ell! (k-\ell+s)! }{s!}.
\end{eqnarray*}
Thus, $a_k=b_k$ for $k \in \Zp$, and \eqref{ass} is proved.
\end{proof}

The elements of the superalgebra $\cA \blb z^{-1} \brb \blb \pz^{-1} \brb$ are called \emph{pseudo-differential operators}.
Let $\cA \zpz$ denote the subspace of $\cA \blb z^{-1} \brb \blb \pz^{-1} \brb$ consisting of all formal series of the form
$$
  \sum_{j=-\infty}^s \sum_{i=-\infty}^r a_{i j} z^i \pz^j, \qquad  r, s \in \Z \quad \text{and} \quad a_{i j} \in \cA.
$$
It is clear that $\cA \zpz$ is closed under multiplication, and so it is a subalgebra of $\cA \blb z^{-1} \brb \blb \pz^{-1} \brb$.
Applying the rules \eqref{rule}, we see that any expression of the form $\disp{\sum_{j=-\infty}^s \sum_{i=-\infty}^r a_{i j} \pz^i z^j}$, where $r, s \in \Z$ and $a_{i j} \in \cA$, is an element of $\cA\zpz$.
In particular,
\begin{equation}\label{rule-2}
\sum_{k=0}^{\infty} (-1)^k \binom{i}{k} \binom{j}{k}  k! \,   \pz^{j-k} \, z^{i-k}=z^{i} \pz^j,  \qquad \mbox{$i, j \in \Z$.}
\end{equation}
Conversely, any element of $\cA \zpz$ can be written as
$\disp{\sum_{j=-\infty}^s \sum_{i=-\infty}^r a_{i j} \pz^i z^j}$ for some $r, s \in \Z$ and $a_{i j} \in \cA$
due to \eqref{rule-2}.

It is easy to see from \eqref{rule} and \eqref{rule-2} that the identity map on the superalgebra $\cA$ extends to a superalgebra automorphism
\begin{equation}\label{pseudo-inv}
\bog: \cA \zpz \longrightarrow \cA \zpz
\end{equation}
such that
$\bog(z)=\pz$ and $\bog(\pz)=-z$. 
Also, $\bog$ has order 4.

Let $z$ and $w$ be commuting even variables.
For later use, we let $\cA\blb z^{-1}, w^{-1} \brb$ denote the superalgebra of all formal series of the form
$$
  \sum_{j=-\infty}^s \sum_{i=-\infty}^r a_{i j} z^i w^j, \qquad  r, s \in \Z \quad \text{and} \quad a_{i j} \in \cA.
$$

\section{Gaudin models with irregular singularities}  \label{GM}

In this section, we recall Takiff superalgebras and introduce evaluation maps of higher orders.
The primary purpose is to discuss Gaudin models with irregular singularities \cite{Ry06, FFTL} and their fundamental properties.

\subsection{Takiff superalgebras} \label{takiff}
For any Lie superalgebra $\fg$, we denote by $U(\fg)$ the universal enveloping algebra of $\fg$.
For any even variable $t$, the loop algebra $\fg[t, t^{-1}]:=\fg\otimes \C[t, t^{-1}]$ is defined to be the Lie superalgebra with commutation relations
$$
\hspace{1cm}  \big[A \otimes t^r, B \otimes t^s \big]=[A, B] \otimes t^{r+s} \qquad \mbox{for $A, B \in \fg$ and $r, s \in \Z$.}
$$
Here $[A, B]$ is the supercommutator of $A$ and $B$.
The current algebra $\fg[t]:=\fg \otimes \C[t]$ is a subalgebra of $\fg[t, t^{-1}]$.
We identify the Lie superalgebra $\fg$ with the subalgebra $\fg\otimes 1$ of constant polynomials in $\fg[t]$ and hence $U(\fg) \subseteq U(\fg[t])$.

For any $\ga \in \N$, $t^\ga \fg[t]:=\fg \otimes t^\ga \C[t]$ is an ideal of $\fg[t]$.
The corresponding quotient Lie superalgebra $\fg^{(\ga)}[t]:=\fg[t]/t^\ga \fg[t]$ is called a (generalized) \emph{Takiff superalgebra} over $\fg$.
For the sake of simplicity, we write
$$
A \otimes \ov{t}^i=A \otimes t^i+t^\ga \fg[t] \qquad \mbox{for $i \in \Zp$}.
$$

\subsection{Evaluation maps of higher orders} \label{ev-map}

Let $\fg$ be a Lie (super)algebra and $a \in \C$. There is an \emph{evaluation homomorphism} $\ev_a: U( \fg[t]) \longrightarrow U(\fg)$ given by
$$
\ev_a (A \otimes t^r )= a^r A \qquad \mbox{for $A \in \fg$ and $r \in \Zp$}.
$$
For $\ga \in \N$, write
$$
t_a=t-a \qquad \text{and} \qquad \fg(a, \ga)=\fg^{(\ga)}[t_a]
$$
There is a Lie superalgebra homomorphism $\fg[t] \longrightarrow \fg(a, \ga)$, defined by
$$
A \otimes t^r \mapsto \sum_{i=0}^r   \binom{r}{i} a^{r-i} (A \otimes \ov{t}_a^i),
$$
for $A \in \fg$ and $r \in \Zp$.
It extends to a superalgebra homomorphism
$$
\ev_{(a, \ga)}: U(\fg[t]) \longrightarrow U(\fg(a, \ga)),
$$
which we call the \emph{evaluation map of order $\ga$} at $a$.
If $\ga=1$, we have the commutative diagram
$$
\begin{tikzcd}
U(\fg[t]) \arrow[rd, "\ev_a"'] \arrow[r, "\ev_{(a,1)}"] & U(\fg(a, 1))  \arrow[d, "\cong"] \\
&  U(\fg)
\end{tikzcd}
$$

Fix $\ell \in \N$. For $\z:=(z_1, \ldots, z_\ell) \in \C^\ell$ and $\jjp:=(\jp_1,\ldots, \jp_{\ell}) \in \N^\ell$,
consider the direct sum of Takiff superalgebras
$$
\gz:=\bigoplus_{i=1}^{\ell} \fg(z_i, \ga_i).
$$
These $z_i$'s correspond to the singularities for the Gaudin algebras to be defined in \secref{Gaudin-gld} and \secref{Gaudin-gls}.

Let $\Delta^{(\ell-1)}: U(\fg[t]) \longrightarrow U(\fg[t])^{\otimes \ell}$ be the $(\ell-1)$-fold coproduct on $U(\fg[t])$.
The composite map
$$
\evz:=(\ev_{z_1} \otimes \ldots \otimes \ev_{z_\ell}) \circ \Delta^{(\ell-1)} : U( \fg[t]) \longrightarrow U(\fg)^{\otimes {\ell}}
$$
is called the \emph{evaluation map at $\z$}.
On the other hand, we also have the map
$$
\evjz:= \lb \ev_{(z_1, \jp_1)} \otimes \ldots \otimes \ev_{(z_\ell, \jp_\ell)} \rb \circ \Delta^{(\ell-1)} : U( \fg[t]) \longrightarrow U(\gz),
$$
called the \emph{evaluation map of order $\jjp$ at $\z$}.
If $\jjp=(1^\ell)$, i.e., $\jp_i=1$ for $1 \le i \le \ell$, then we have the commutative diagram
$$
\begin{tikzcd}
U(\fg[t]) \arrow[rd, "\evz"'] \arrow[r, " \un{\ev}_{\lb \z, (1^\ell) \rb}"] & U \lb \fg \! \lb \z, (1^\ell) \rb \rb \arrow[d, "\cong"] \\
&  U(\fg)^{\otimes \ell}
\end{tikzcd}
$$

\subsection{Feigin--Frenkel centers}\label{FFC}

Let $t$ be an even variable.
For a general linear Lie (super)algebra $\fg$, let $\Vac(\fg)$ be the \emph{universal affine vertex algebra at the critical level} associated to the affine Lie (super)algebra $\wh{\fg}:=\fg[t,t^{-1}]\oplus \C K$; see \cite{Fr07, FBZ, Mo18, MR} for details. Note that $\Vac(\fg)$ is a $\wh{\fg}$-module.
 The center of the vertex algebra $\Vac(\fg)$ is given by
$$
\fzg=\setc*{\! v \in \Vac(\fg)}{\fg[t] v=0 \!},
$$
called the \emph{Feigin--Frenkel center}.
By the Poincar\'e--Birkhoff--Witt theorem, we identify $\Vac(\fg)$ with $U(\fgm)$ (as superspaces).
Moreover, the (super)algebra $U(\fgm)$ is equipped with the (even) derivation $\T:=-d/dt$ defined by
\begin{equation}\label{der}
\T(1)=0 \qquad \text{and}\qquad \T(A \otimes t^{-r})=r A \otimes t^{-r-1}
\end{equation}
for $A \in \fg$ and $r \in \N$.
Note that $\T$ corresponds to the translator operator on $\Vac(\fg)$, and that $\fzg$ can be viewed as a commutative subalgebra of $U(\fgm)$ and is $\T$-invariant.

Suppose $\mu \in \fg^*$ vanishes on the odd part of $\fg$. Then there is a superalgebra homomorphism
$$
 \Psi^\mu :  U(\fgm)   \longrightarrow U(\fg[t])  [\tns[z^{-1}]\tns]
$$
given by
$$
 \hspace{1cm}  \Psi^\mu(A \otimes t^{-r})= A \otimes (t-z)^{-r} +\de_{r,1} \mu(A), \quad \mbox{for $A \in \fg \,$ and $\, r \in \N$.}
$$
The evaluation map $\evjz$ extends to
$$
\evjz: U( \fg[t]) [\tns[ z^{-1}]\tns]   \longrightarrow U(\gz) [\tns[ z^{-1}]\tns]
$$
by setting $\evjz(z^{-1})=z^{-1}$.
We have the composite map
$$
\Pmuz :=\evjz \circ \Psi^\mu: U(\fgm) \longrightarrow U \big(\gz \big) [\tns[z^{-1}]\tns].
$$
Using the identity
 $$
(t-z)^{-1}=-\frac{1}{z-z_i} \left(1-\frac{\tzi}{z-z_i} \rb^{-1}=-\sum_{k=0}^\infty \frac{\tzi^k}{(z-z_i)^{k+1}}
 $$
 for $1 \le i \le \ell$,
we find that
\begin{equation} \label{Pmuz}
\Pmuz \lb A \otimes t^{-1} \rb= - \sum_{i=1}^{\ell}  \sum_{k=0}^{\jp_i-1} \frac{A \otimes \otzi^k}{(z-z_i)^{k+1}}+ \mu(A),  \qquad \mbox{for $A \in \fg$.}
\end{equation}

Let
$$
\tau=-\partial_t,
$$
where $t$ and $\partial_t$ satisfy the rules similar to \eqref{rule}.
The map $\Pmuz$ extends to a superalgebra homomorphism
$$
\Pmuz :  U(\fgm) \blb \tau^{-1} \brb  \longrightarrow U \big(\gz \big) \zpz
$$
defined by
 $$
\Pmuz   \left(\sum_{i=-\infty}^r a_i \tau^i \rb=\sum_{i=-\infty}^r \Pmuz  (a_i) \pz^i
 $$
for $a_i \in U(\fgm)$ and $r \in \Z$.

\subsection{Gaudin algebras for $\gld$} \label{Gaudin-gld}

Let
$$
\cT_d = \Big[\delta_{i, j} \tau+e_{ij}\otimes t^{-1} \Big]_{i,j=1,\ldots,d},
$$
which is a $d \times d$ Manin matrix over $U(\gldm)[\tau]$.
We have an expansion
$$
 \cdet (\cT_d)=\tau^d+\sum_{i=1}^{d}a_i \tau^{i-1}
$$
for some $a_1, \ldots, a_d \in \fzd$.

\begin{thm} [{\cite[Theorem 3.1]{CM} (cf. \cite{CT})}] \label{complete-SS}
The set $\{\T^r a_i \, | \, i=1, \ldots, d, \, r \in \Zp \}$ is algebraically independent, and
$\fzd=\C \! \ls \T^r a_i \, | \, i=1, \ldots, d, \, r \in \Zp \, \rs.$
\end{thm}

Fix $\z:=(z_1, \ldots, z_\ell) \in \C^\ell$ and $\jjp:=(\jp_1, \ldots, \jp_\ell) \in \N^\ell$.
For any $\mu \in \gld^*$, let
$$
\cL_d^{\mu}(\z, \jjp)  = \Big[\Pmuz \big(\delta_{i, j} \tau+e_{ij}\otimes t^{-1} \big)\Big]_{i,j=1,\ldots,d}.
$$
Since $\Pmuz$ is a homomorphism, we have
$$
 \cdet (\cL_d^{\mu}(\z, \jjp) )=\pz^d+\sum_{i=1}^{d}a_i(z)\pz^{i-1},
$$
where $a_i(z):=\Pmuz (a_i) \in U \big(\gldz \big) [\tns[z^{-1}]\tns]$.

Let $\cA_d^{\mu}(\z, \jjp)$ be the subalgebra of $U \big(\gldz \big)$ generated by the coefficients of the series $a_i(z)$ for $i=1, \ldots, d$.
We call $\cA_d^{\mu}(\z, \jjp)$ the \emph{Gaudin algebra for $\gld$ with singularities of orders $\jp_i$ at $z_i$, $i=1, \ldots, \ell$, relative to $\mu$}; see \eqref{Pmuz}. For $1 \le i \le \ell$, the singularity at $z_i$ is said to be \emph{regular} if $\jp_i=1$ and \emph{irregular} otherwise.
The commutativity of $\fzg$ yields the commutativity of $\Pmuz \lb \fzg \rb$, and hence the algebra $\cA_d^{\mu}(\z, \jjp)$ is commutative \cite{FFTL, Ry06}.

\begin{rem} \label{complete}
 The algebra  $\cA_d^{\mu}(\z, \jjp)$ coincides with the subalgebra of $U \big(\gldz \big)$ generated by the coefficients of $\Pmuz(S)\in U \big(\gldz \big) [\tns[z^{-1}]\tns]$ for $S \in \fzd$.
 This follows from \thmref{complete-SS} and $\Pmuz \big(\T(S) \big)=d/dz \big(\Pmuz(S) \big)$ (cf. \cite[Proposition 5.9]{ChL25-1}).
 \end{rem}

 For any $\cA_d^{\mu}(\z, \jjp)$-module $V$, let $\cA_d^{\mu}(\z, \jjp)_V$ denote the image of the Gaudin algebra $\cA_d^{\mu}(\z, \jjp)$ in $\End(V)$. It is called the \emph{Gaudin algebra} of $V$ with singularities of orders $\jp_i$ at $z_i$, $i=1, \ldots, \ell$, relative to $\mu$.

As in \eqref{pseudo-inv}, the identity map on the algebra $U \big(\gldz \big)$ extends to an automorphism
$$
\bog: U \big(\gldz \big) \zpz \longrightarrow U \big(\gldz \big) \zpz
$$
such that
$\bog(z)=\pz$ and $\bog(\pz)=-z$.
Evidently,
$$
\cL_d^{\mu}(\z, \jjp)  = \Big[\delta_{i, j} \pz+\Pmuz \big(e_{ij}\otimes t^{-1} \big)\Big]_{i,j=1,\ldots,d}.
$$
Set
\begin{equation} \label{wtL}
\wh{\cL}_d^{\mu}(\z, \jjp) = - \Big[\bog \big( \delta_{i, j}  \pz - \Psi^{\mu}_{(\z, \jjp)} \big(e_{ji}\otimes t^{-1} \big) \Big]_{i,j=1,\ldots,d}.
\end{equation}
We have the following (see \cite[the discussion in Section 4.2]{VY}).

\begin{prop} \label{L'}
The column determinant $\cdet \big(\wh{\cL}_d^{\mu}(\z, \jjp) \big)$ has an expansion 
$$\cdet \big(\wh{\cL}_d^{\mu}(\z, \jjp) \big)=z^d+\sum_{i=1}^{d} \wh{a}_i(\pz) z^{i-1}$$ 
for some $\wh{a}_i(\pz) \in U \big(\gldz \big) [\tns[\pz^{-1}]\tns]$.
Moreover, the algebra $\cA_d^{\mu}(\z, \jjp)$ is generated by the coefficients of the series $\wh{a}_i(\pz)$ for $i=1, \ldots, d$.
\end{prop}

\subsection{Gaudin algebras for $\gl_\pqmn$} \label{Gaudin-gls}

Recall the symbol $\ds=\pqmn$ and fix $\s:=(0^p, 1^q, 0^m, 1^n) \in \cS_\ds$.
Let
$$
\cT_\ds=\Big[\delta_{i,j}\tau+(-1)^{|i|}E^i_j \otimes t^{-1}\Big]_{i,j\in \I}.
$$
We know that $\cT_\ds$ is an amply invertible $(p+q+m+n) \times (p+q+m+n)$ Manin matrix of type $\s$ over $U(\glsm) \blb \tau^{-1} \brb$.
We have an expansion
$$
\Bers  \! \left(\cT_\ds \rb=\sum_{i=-\infty}^{p+m-q-n}b_i \tau^i,
$$
for some $b_i \in \fzr$; see \cite[Corollary 3.3]{MR} and \cite[Section 3.1]{ChL25-1}. 
Define $\hat \fz_\ds$ to be the subalgebra of $U(\glsm)$ generated by
$$
\{\T^r b_i \, | \, i \le p+m-q-n, \, i \in \Z, \, r \in \Zp \}.
$$
Since all $b_i$'s belong to $\fzr$, and $\fzr$ is $\T$-invariant, we see that $\hat \fz_\ds$ is a commutative subalgebra of $\fzr$.

Fix $\z \in \C^\ell$ and $\jjp \in \N^\ell$.
For any $\mu \in \gls^*$ which vanishes on the odd part of $\gls$, let
\begin{equation} \label{Lsmu}
\Lr^{\mu}(\z, \jjp) = \Big[\Pmuz   \big(\delta_{i,j}\tau+(-1)^{|i|}E^i_j \otimes t^{-1} \big) \Big]_{i,j\in \I}.
\end{equation}
Then
$$
\Bers \! \left( \Lr^{\mu}(\z, \jjp)  \rb=\sum_{i=-\infty}^{p+m-q-n} b_i (z) \pz^i,
$$
which belongs to $U \big(\glsz \big) \zpz$. Here $b_i (z):=\Pmuz   (b_i)$.

Let $\cA_\ds^{\mu}(\z, \jjp)$ be the subalgebra of $U \big(\glsz\big)$ generated by the coefficients of the series $b_i (z)$, for $i \in \Z$ with $i \le p+m-q-n$.
We call $\cA_\ds^{\mu}(\z, \jjp)$ the \emph{Gaudin algebra for $\gls$ with singularities of orders $\jp_i$ at $z_i$, $i=1, \ldots, \ell$, relative to $\mu$}.
Also, for $1 \le i \le \ell$, the singularity at $z_i$ is said to be \emph{regular} if $\jp_i=1$ and \emph{irregular} otherwise.

\begin{rem} \label{complete-s}
\begin{enumerate}[\normalfont(i)]

\item The algebra $\cA_\ds^{\mu}(\z, \jjp)$ equals the subalgebra of $U \big(\glsz \big)$ generated by the coefficients of $\Pmuz(S)$ for $S \in \hat \fz_\ds$ (as $\Pmuz \big( \T(S) \big)=d/dz \big( \Pmuz(S) \big)$) (cf. \cite[Proposition 5.9]{ChL25-1}).

\item  We do not know whether \remref{complete} holds for $\gls$ in general as it is still a conjecture that $\fzr=\hat \fz_\ds$.
(\emph{Note}: The conjecture has been shown to be valid for $\gl_{1|1}$ \cite{MM15} and $\gl_{2|1}$ \cite{AN}.)

\end{enumerate}
\end{rem}

The commutativity of $\hat \fz_\ds$ yields the commutativity of $\Pmuz \lb \hat \fz_\ds \rb$.
Using an argument similar to the proof of \cite[Corollary 3.6]{MR}, we see that $\cA_\ds^{\mu}(\z, \jjp)$ is a commutative algebra.

For any $\cA_\ds^{\mu}(\z, \jjp)$-module $V$, let $\cA_\ds^{\mu}(\z, \jjp)_V$ denote the image of the Gaudin algebra $\cA_\ds^{\mu}(\z, \jjp)$ in $\End(V)$. We call $\cA_\ds^{\mu}(\z, \jjp)_V$ the \emph{Gaudin algebra} of $V$ with singularities of orders $\jp_i$ at $z_i$, $i=1, \ldots, \ell$, relative to $\mu$.

\section{A duality for Gaudin models} \label{duality}

In this section, we construct a joint action of the Gaudin algebras with irregular singularities for $\gld$ and $\gl_\pqmn$ on the Fock space of $d(p+m)$ bosonic and $d(q+n)$ fermionic oscillators. We establish a duality of $(\gld, \gl_\pqmn)$ for Gaudin models and give an application to the action of the Gaudin algebra for $\gl_\pqmn$ on a certain class of modules over a direct sum of Takiff superalgebras over $\gl_\pqmn$.

\subsection{Actions on Fock spaces} \label{Fock}

Let $\F$ be the polynomial superalgebra generated by the variables $\xai$ and $\yar$, for $i=1, \ldots, m+n$, $r=1, \ldots, p+q$ and $a=1, \ldots, d$, where $\xai$ (resp., $\yar$) are even for $1 \le i \le m$ (resp., $1 \le r \le p$) and are odd otherwise.
There are commuting actions of $\gld$ and $\gl_\pqmn$ on $\F$ that form a Howe dual pair \cite{CLZ}.
The superalgebra $\F$ may be regarded as the Fock space of $d(p+m)$ bosonic and $d(q+n)$ fermionic oscillators (cf. \cite{CL03, LZ06}).
Let $\cD$ be the corresponding \emph{Weyl superalgebra}. That is, $\cD$ is the associative unital superalgebra generated by $\xai$ and $\yar$ as well as their derivatives
$$
\pxai:=\frac{\pa}{\pa x^a_i} \qquad \text{and} \qquad \pyar:=\frac{\pa}{\pa y^a_r}
$$
for $1 \le i \le m+n$, $1 \le r \le p+q$ and $1 \le a \le d$.
The superalgebra $\cD$ is naturally a Lie superalgebra with supercommutator $[ \cdot, \cdot ]$ defined as in \eqref{s-comm}.
Clearly, $\cD$ acts naturally on $\F$, and so it is a subalgebra of $\End(\F)$.

Fix $\dd \in \N$ and $p^\prime, q^\prime, m^\prime, n^\prime \in \Zp$ with $\dd \le d$, $p^\prime \le p$, $q^\prime \le q$, $m^\prime \le m$, and $n^\prime \le n$. Let
$$
\el=p^\prime+q^\prime+m^\prime+n^\prime.
$$
Fix two sequences
$$
\w:=(w_1,\ldots,w_{\dd}) \in \C^{\dd} \qquad \text{and} \qquad \z:=(z_1,\ldots, z_{\el}) \in \C^{\el}.
$$
We write $d$ (resp., $p$, $q$, $m$ and $n$) into a sum of $\dd$ (resp., $p^\prime$, $q^\prime$, $m^\prime$ and $n^\prime$) positive integers.
More precisely, we let $\bjp_1, \ldots, \bjp_{\dd} \in \N$ be such that
$$
\sum_{i=1}^{\dd} \bjp_i =d,
$$
and
let $\jp_1, \ldots, \jp_{\el} \in \N$ be such that
\begin{equation} \label{rel}
\sum_{i=1}^{p^\prime} \jp_i = p, \quad \sum_{i=1}^{q^\prime} \jp_{p^\prime+i}= q, \quad \sum_{i=1}^{m^\prime} \jp_{p^\prime+q^\prime+i} = m, \quad \sum_{i=1}^{n^\prime} \jp_{p^\prime+q^\prime+m^\prime+i} = n.
\end{equation}
Set
$$
\bjjp=(\bjp_1,\ldots, \bjp_{\dd}) \qquad \text{and} \qquad  \jjp=(\jp_1,\ldots, \jp_{\el}).
$$
Also, we define $\va_1, \ldots, \va_{\dd+1} \in \Zp$ by
$$
\va_1=0 \qquad \text{and} \qquad  \va_{i+1}=\sum_{j=1}^i \bjp_j, \quad 1 \le i \le \dd,
$$
and define $\rva_1, \ldots, \rva_{\el+1} \in \Zp$ by
$$
\rva_1=0  \qquad \text{and} \qquad  \rva_{i+1}=\sum_{j=1}^i \jp_j, \quad1 \le i \le \el.
$$

For $1\le i \le p^\prime+q^\prime$ and $p^\prime+q^\prime+1 \le j \le \el$, let
$$
\y_{(i)}=\{ y^a_r \, | \, \rva_i+1 \le r \le \rva_{i+1}, 1 \le a \le d \}
$$
and
$$
\x_{(j)}=\{ x^a_r \, | \, \rva_{j}+1-p-q \le r \le \rva_{j+1}-p-q, 1 \le a \le d \}.
$$
We have
\begin{equation} \label{Fock-pqmn}
\F =  \lb \bigotimes_{i=1}^{p^\prime+q^\prime} \C \big[ \y_{(i)} \big] \rb \otimes \lb \bigotimes_{j=p^\prime+q^\prime+1}^{\el} \C \big[\x_{(j)} \big] \rb \!.
\end{equation}
The following proposition induces an action of $U \big(\gldzb \big)$ on $\F$.

\begin{prop} \label{gld-map}

There is a Lie superalgebra homomorphism $\vpd : \gldzb \longrightarrow \cD$
defined by
\begin{equation} \label{vpd-e}
{\everymath={\disp}
\begin{array}{cll}
e_{ab}\otimes \otzi^k \mapsto -\sum_{r=\rva_i+1}^{\rva_{i+1}-k}  y^b_r \pa_{y^a_{r+k}},  \\
e_{ab}\otimes \otzj^l \mapsto \sum_{r=\rva_{j}+1-p-q}^{\rva_{j+1}-l-p-q} (-1)^{|p+q+r|}  \pa_{x^b_r} x^a_{r+l} ,
\end{array}}
\end{equation}
for  $0 \le k \le \jp_i-1$, $0 \le l \le \jp_j-1$, $1 \le i \le p^\prime+q^\prime$, $p^\prime+q^\prime+1 \le j \le \el$ and $1 \le a, b \le d$.
Thus, $\vpd$ extends to a superalgebra homomorphism
$\vpd : U \big(\gldzb \big) \longrightarrow \cD$.

\end{prop}
\begin{proof}
It suffices to verify that $\vpd([A, B])=[\vpd (A), \vpd (B)]$ for all elements $A$ and $B$ on the left-hand side of \eqref{vpd-e}.
Clearly,
for  $0 \le k \le \jp_i-1$, $0 \le l \le \jp_j-1$, $1 \le i \le p^\prime+q^\prime$, $p^\prime+q^\prime+1 \le j \le \el$ and $1 \le a, b, a', b' \le d$,
$$
\vpd \! \lb  \ls e_{ab}\otimes \otzi^k, e_{a' b'}\otimes  \otzj^l \rs \rb=0 = \ls\vpd \! \lb  e_{ab}\otimes \otzi^k \rb, \vpd \! \lb  e_{a' b'}\otimes  \otzj^l \rb \rs.
$$
Now for $0 \le k \le \jp_i-1$, $0 \le l \le \jp_j-1$, $1 \le i, j \le p^\prime+q^\prime$ and $1 \le a, b, a', b' \le d$, we have
\begin{eqnarray*}
\vpd \! \lb  \ls e_{ab}\otimes \otzi^k, e_{a' b'}\otimes \otzj^l \rs \rb
&=&\de_{i,j} \vpd \! \lb  \ls e_{ab}, e_{a' b'} \rs \otimes \otzi^{k+l} \rb \\
&=&\de_{i,j} \vpd \! \lb  (\de_{a',b}  e_{a b'}-\de_{a,b'} e_{a' b}) \otimes \otzi^{k+l} \rb \\
&=&\de_{i,j}  \lb - \de_{a',b} \sum_{r=\rva_i+1}^{\rva_{i+1}-k-l}  y^{b'}_r \pa_{y^a_{r+k+l}} + \de_{a,b'} \sum_{r=\rva_i+1}^{\rva_{i+1}-k-l}  y^b_r \pa_{y^{a'}_{r+k+l}} \rb.
\end{eqnarray*}
Meanwhile,
\begin{eqnarray*}
& & \ls \vpd(e_{ab}\otimes \otzi^k), \vpd(e_{a' b'}\otimes \otzj^l)\rs \\
& & \hspace{1.5cm} =  \ls \sum_{r=\rva_i+1}^{\rva_{i+1}-k}  y^b_r \pa_{y^a_{r+k}}, \sum_{s=\rva_j+1}^{\rva_j+\jp_j-l}  y^{b'}_s \pa_{y^{a'}_{s+l}} \rs \\
& & \hspace{1.5cm} =  \de_{i,j} \sum_{r=\rva_i+1}^{\rva_{i+1}-k} \sum_{s=\rva_i+1}^{\rva_{i+1}-l}  \ls  y^b_r \pa_{y^a_{r+k}},  y^{b'}_s \pa_{y^{a'}_{s+l}} \, \rs\\
& & \hspace{1.5cm} =  \de_{i,j} \sum_{r=\rva_i+1}^{\rva_{i+1}-k} \sum_{s=\rva_i+1}^{\rva_{i+1}-l}  \lb y^b_r \pa_{y^a_{r+k}} y^{b'}_s \pa_{y^{a'}_{s+l}} - y^{b'}_s \pa_{y^{a'}_{s+l}} y^b_r \pa_{y^a_{r+k}} \rb \\
& & \hspace{1.5cm} =  \de_{i,j} \sum_{r=\rva_i+1}^{\rva_{i+1}-k} \sum_{s=\rva_i+1}^{\rva_{i+1}-l}   \lb y^b_r  \ls \pa_{y^a_{r+k}}, y^{b'}_s \, \rs \pa_{y^{a'}_{s+l}} - y^{b'}_s \ls \pa_{y^{a'}_{s+l}}, y^b_r \, \rs \pa_{y^a_{r+k}} \rb \\
& & \hspace{1.5cm} =  \de_{i,j} \lb \sum_{r=\rva_i+1}^{\rva_{i+1}-k-l} \de_{a,b'} y^b_r  \pa_{y^{a'}_{r+k+l}} - \sum_{s=\rva_i+1}^{\rva_{i+1}-k-l} \de_{a',b} y^{b'}_s  \pa_{y^a_{s+k+l}} \rb.
\end{eqnarray*}
This gives
$$
\vpd \! \lb  \ls e_{ab}\otimes \otzi^k, e_{a' b'}\otimes \otzj^l \rs \rb
=\ls \vpd \! \lb  e_{ab}\otimes \otzi^k), \vpd(e_{a' b'}\otimes \otzj^l \rb \rs.
$$
We can show similarly that for
$0 \le k \le \jp_i-1$, $0 \le l \le \jp_j-1$, $p^\prime+q^\prime+1 \le i, j \le \el$ and $1 \le a, b, a', b' \le d$,
$$
\vpd \! \lb  \ls e_{ab}\otimes t_{z_{i}}^k, e_{a' b'}\otimes \otzj^l \rs \rb
=\ls \vpd \! \lb  e_{ab}\otimes t_{z_{i}}^k \rb, \vpd \! \lb  e_{a' b'}\otimes \otzj^l \rb \rs.
$$
This completes the proof.
\end{proof}

Recall the symbol $\ds=\pqmn$.
For $1 \le a \le \dd$, define
\begin{equation} \label{sigma}
\Sig^{(a)}=\setc*{x^\al_{i}, y^\al_{r}}{ d_a+1 \le \al \le d_{a+1}, 1\le i \le m+ n, 1\le r \le p+q}.
\end{equation}
We have
\begin{equation} \label{Fock-d}
\F = \bigotimes_{a=1}^\dd  \C [\Sig^{(a)} ].
\end{equation}
The following proposition induces an action of $U \big(\glsw \big)$ on $\F$.

\begin{prop} \label{glr-map}
There is a Lie superalgebra homomorphism $\vps : \glsw \longrightarrow \cD$ defined by
\begin{equation} \label{vps-e}
{\everymath={\disp}
\begin{array}{cll}
E^r_s\otimes \otwa^k &\mapsto \sum_{\al=\va_a+1}^{\va_{a+1}-k}  (-1)^{|s|+1}  \pa_{y^{\al+k}_r} y^\al_s, \\
E^r_{p+q+j}\otimes \otwa^k &\mapsto \sum_{\al=\va_a+1}^{\va_{a+1}-k}  \pa_{y^{\al+k}_r} \pa_{x^\al_j}, \\
E^{p+q+i}_s \otimes \otwa^k &\mapsto \sum_{\al=\va_a+1}^{\va_{a+1}-k}   (-1)^{|s|+1} x^{\al+k}_i y^\al_s,\\
E^{p+q+i}_{p+q+j}\otimes \otwa^k  &\mapsto \sum_{\al=\va_a+1}^{\va_{a+1}-k}   x^{\al+k}_i \pa_{x^\al_j},
\end{array}}
\end{equation}
for $0 \le k \le \bjp_a-1$, $1 \le a \le \dd$, $1 \le r, s \le p+q$ and $1 \le i, j \le m+n$. Thus, $\vps$ extends to a superalgebra homomorphism
$\vps : U \big(\glsw \big) \longrightarrow \cD$.

\end{prop}

\begin{proof}

We need to verify that $\vps([A, B])=\ls \, \vps (A), \vps (B) \, \rs$ for all elements $A$ and $B$ on the left-hand side of \eqref{vps-e}.
Clearly, for $0 \le k \le \bjp_a-1$, $0 \le l \le \bjp_b-1$, $1 \le a, b \le \dd$, $1 \le r, s \le p+q$ and $1 \le i, j \le m+n$,
$$
\vps \! \lb  \ls E^r_s\otimes \otwa^k, E^{p+q+i}_{p+q+j}\otimes \otwa^l \rs \rb=0 = \ls \vps \! \lb  E^r_s\otimes \otwa^k \rb, \vps \! \lb  E^{p+q+i}_{p+q+j}\otimes \otwa^l \rb \rs.
$$
We will only verify that for $0 \le k \le \bjp_a-1$, $0 \le l \le \bjp_b-1$, $1 \le a, b \le \dd$, $1 \le r, s, s' \le p+q$ and $1 \le i, j \le m+n$,
\begin{equation} \label{ovphi-e1}
\vps \! \lb  \ls E^r_s\otimes \otwa^k , E^{s'}_{p+q+j} \otimes \otwb^l  \rs \rb= \ls\vps \! \lb  E^r_s\otimes \otwa^k \rb, \vps \! \lb  E^{s'}_{p+q+j} \otimes \otwb^l \rb \rs
\end{equation}
and
\begin{equation} \label{ovphi-e2}
\vps \! \lb  \ls E^r_{p+q+j}\otimes \otwa^k , E^{p+q+i}_s \otimes \otwb^l \rs \rb= \ls\vps \! \lb  E^r_{p+q+j}\otimes \otwa^k \rb, \vps \! \lb  E^{p+q+i}_s \otimes \otwb^l \rb \rs \!.
\end{equation}
The verification of all other identities can be carried out analogously.
Write $|\ov{i}|=|p+q+i|$ for $1 \le i \le m+n$.
Let us prove \eqref{ovphi-e1}. We have
\begin{eqnarray*}
\vps \! \lb  \ls E^r_s\otimes \otwa^k , E^{s'}_{p+q+j} \otimes \otwb^l  \rs  \rb
&=&\de_{a,b} \vps \! \lb  \ls E^r_s, E^{s'}_{p+q+j} \rs \otimes \otwa^{k+l} \rb \\
&=&\de_{a,b} \de_{s, s'} \vps \! \lb  E^r_{p+q+j}  \otimes \otwa^{k+l} \rb\\
&=&\de_{a,b} \de_{s, s'}  \sum_{\al=\va_a+1}^{\va_{a+1}-k-l}  \pa_{y^{\al+k+l}_r} \pa_{x^\al_j}.
\end{eqnarray*}
Meanwhile,
\begin{eqnarray*}
& & \ls \vps(E^r_s\otimes \otwa^k), \vps(E^{s'}_{p+q+j} \otimes \otwb^l )\rs \\
&=& \ls \sum_{\al=\va_a+1}^{\va_{a+1}-k}  (-1)^{|s|+1}  \pa_{y^{\al+k}_r} y^\al_s, \sum_{\be=\va_b+1}^{\va_{b+1}-l}  \pa_{y^{\be+l}_{s'}} \pa_{x^\be_j} \rs \\
&=& \de_{a,b}\sum_{\al=\va_a+1}^{\va_{a+1}-k} \sum_{\be=\va_a+1}^{\va_{a+1}-l}   (-1)^{|s|+1}  \ls  \pa_{y^{\al+k}_r} y^\al_s,  \pa_{y^{\be+l}_{s'}} \pa_{x^\be_j} \rs\\
&=& \de_{a,b}\sum_{\al=\va_a+1}^{\va_{a+1}-k} \sum_{\be=\va_a+1}^{\va_{a+1}-l}   (-1)^{|s|+1}  \lb \pa_{y^{\al+k}_r} y^\al_s \pa_{y^{\be+l}_{s'}} \pa_{x^\be_j}
- (-1)^{(|r|+|s|)(|s'|+|\ov{j}|)}\pa_{y^{\be+l}_{s'}} \pa_{x^\be_j}  \pa_{y^{\al+k}_r} y^\al_s  \rb \\
&=& \de_{a,b}\sum_{\al=\va_a+1}^{\va_{a+1}-k} \sum_{\be=\va_a+1}^{\va_{a+1}-l}   (-1)^{|s|+1}  \lb \pa_{y^{\al+k}_r} \ls y^\al_s,  \pa_{y^{\be+l}_{s'}} \rs \pa_{x^\be_j}  \rb\\
&=& \de_{a,b} \de_{s, s'} \sum_{\be=\va_a+1}^{\va_{a+1}-k-l}  \pa_{y^{\be+k+l}_r} \pa_{x^\be_j}.
\end{eqnarray*}
This completes the proof of \eqref{ovphi-e1}.
Now we prove \eqref{ovphi-e2}. We find that
\begin{eqnarray*}
& & \vps \! \lb  \ls E^r_{p+q+j}\otimes \otwa^k , E^{p+q+i}_s \otimes \otwb^l \rs \rb \\
&=&\de_{a,b} \vps \! \lb  \ls E^r_{p+q+j}, E^{p+q+i}_s \, \rs \otimes \otwa^{k+l} \rb \\
&=&\de_{a,b} \vps \! \lb   \lb \de_{i, j} E^r_s- (-1)^{(|r|+|\ov{j}|)(|s|+|\ov{i}|)}  \de_{r, s}  E^{p+q+i}_{p+q+j} \rb \otimes \otwa^{k+l} \rb\\
&= &  \de_{a,b} \Big( \de_{i, j} \sum_{\al=\va_a+1}^{\va_{a+1}-k-l}  (-1)^{|s|+1}  \pa_{y^{\al+k+l}_r} y^\al_s - (-1)^{(|r|+|\ov{j}|)(|s|+|\ov{i}|)}  \de_{r, s} \sum_{\be=\va_a+1}^{\va_{a+1}-k-l}   x^{\be+k+l}_i \pa_{x^\be_j} \Big).
\end{eqnarray*}
Meanwhile,
\begin{eqnarray*}
& & \ls  \vps( E^r_{p+q+j}\otimes \otwa^k),  \vps (E^{p+q+i}_s \otimes \otwb^l) \rs \\
&=& \ls  \sum_{\al=\va_a+1}^{\va_{a+1}-k}  \pa_{y^{\al+k}_r} \pa_{x^\al_j}, \sum_{\be=\va_b+1}^{\va_{b+1}-l}   (-1)^{|s|+1} x^{\be+l}_i y^\be_s \rs \\
&=& \de_{a,b}\sum_{\al=\va_a+1}^{\va_{a+1}-k} \sum_{\be=\va_a+1}^{\va_{a+1}-l}   (-1)^{|s|+1}  \ls  \pa_{y^{\al+k}_r} \pa_{x^\al_j},   x^{\be+l}_i y^\be_s \rs\\
&=& \de_{a,b}\sum_{\al=\va_a+1}^{\va_{a+1}-k} \sum_{\be=\va_a+1}^{\va_{a+1}-l}   (-1)^{|s|+1} \lb \pa_{y^{\al+k}_r} \pa_{x^\al_j} x^{\be+l}_i y^\be_s -  (-1)^{(|r|+|\ov{j}|)(|s|+|\ov{i}|)}x^{\be+l}_i y^\be_s  \pa_{y^{\al+k}_r} \pa_{x^\al_j}  \rb \\
&=& \de_{a,b}\sum_{\al=\va_a+1}^{\va_{a+1}-k} \sum_{\be=\va_a+1}^{\va_{a+1}-l}   (-1)^{|s|+1} \Big( \pa_{y^{\al+k}_r} \ls \pa_{x^\al_j}, x^{\be+l}_i  \, \rs y^\be_s \\
& & \hspace{6cm} - (-1)^{(|r|+|\ov{j}|)(|s|+|\ov{i}|)}x^{\be+l}_i \ls y^\be_s,  \pa_{y^{\al+k}_r} \rs \pa_{x^\al_j}  \Big) \\
&=& \! \de_{a,b} \Big( \de_{i,j} \sum_{\be=\va_a+1}^{\va_{a+1}-k-l}  (-1)^{|s|+1}  \pa_{y^{\be+k+l}_r} y^\be_s \\
& & \hspace{4.5cm}  - (-1)^{(|r|+|\ov{j}|)(|s|+|\ov{i}|)}  \de_{r, s} \sum_{\al=\va_a+1}^{\va_{a+1}-k-l}   x^{\al+k+l}_i \pa_{x^\al_j} \Big) \!.
\end{eqnarray*}
This completes the proof of \eqref{ovphi-e2}.
\end{proof}

\sloppy
The pair $(\vpd, \vps)$ of homomorphisms induce a joint action of $\Adzw$ and $\Aswz$ on the Fock space $\F$.
We will see, in \secref{gld-gls-duality}, that $(\vpd, \vps)$ gives a duality of $(\gld, \gls)$, which is an equivalence between the actions of $\Adzw$ and $\Aswz$ on $\F$.

\begin{rem}
In the special case where $\bjjp=(1^\dd)$ and $\jjp=(1^\el)$, a variant of the pair $(\vpd, \vps)$ was considered in \cite[pp. 789--790]{CLZ} and \cite[Section 4.3]{ChL25-2}. 
It gives rise to the Howe duality of $(\gld, \gls)$ (see \cite[Theorem 3.3]{CLZ}).
\end{rem}

\subsection{Examples}\label{ex}
In this subsection, we will give examples of $\glsw$-modules that will be used in \secref{app}.

Let 
$$
\dds=\sum_{a=1}^{d} \lb -\sum_{r=1}^{p+q}  y^a_r \pa_{y^{a}_r}+\sum_{i=1}^{m+n} x^{a}_i \pa_{x^a_i} \rb \!.
$$ 
For any monomial $f \in \F$, we have $\dds(f)=k f$ for some $k \in \Z$, and we define the degree of $f$ to be $k$. 
In particular, the degree of each ${x^a_i}$ is $1$, and the degree of each $y^a_r$ is $-1$. For this reason, it is natural to call $\dds$ the degree operator on $\F$.

By definition, $\gl_\ds^{(d)}[t]=\gls(0, d)$, and by \propref{glr-map}, the Fock space $\F$ is a $\gltd$-module.
For each $k\in\Z$, let $V^{(d)}_k$ denote the subspace of $\F$ spanned by monomials of degree $k$. Since $\dds$ commutes with the action of $\gltd$ on $\F$, the space $V^{(d)}_k$ is a $\gltd$-module. 
We have a decomposition of $\F$ into a direct sum of $\gltd$-modules
\begin{equation} \label{F-decomp-1}
\F=\bigoplus_{k\in \Z}V^{(d)}_k.
\end{equation}

For $1 \le a \le \dd$, recall $\Sig^{(a)}$ defined in \eqref{sigma}. The Fock space $\C [\Sig^{(a)}]$ is a $\glswa$-module. Analogous to \eqref{F-decomp-1}, there is also a direct sum decomposition of $\glswa$-modules
$$
\C [\Sig^{(a)}]=\bigoplus_{k\in \Z} V_k^{(\xi_a)},
$$
where $V_k^{(\xi_a)}$ is the subspace of $\C [\Sig^{(a)} ]$ spanned by monomials of degree $k$. 
Define 
\begin{equation} \label{Vk}
\Vk=V^{(\xi_1)}_{k_1}  \otimes \cdots \otimes V^{(\xi_{\dd})}_{k_{\dd}} \qquad  \mbox{for \, $k_1, \ldots, k_{\dd} \in \Z$}.
\end{equation}
By \propref{glr-map} and \eqref{Fock-d}, $\Vk$ is a $\glsw$-module, and we have a direct sum decomposition of $\glsw$-modules
\begin{equation} \label{F-decomp}
\F= \bigoplus_{k_1, \ldots, k_{\dd} \in \Z}  \Vk.
\end{equation}

\subsection{A duality of $(\gld, \gl_\pqmn)$} \label{gld-gls-duality}
The $k \times k$ matrix
$$
\J_{k}(\la):=
\begin{bmatrix}
\la & 1 & \ldots & 0\\
0 & \la & \ddots & \vdots \\
\vdots & \ddots & \ddots & 1\\
0 & \ldots & 0 & \la
\end{bmatrix}
$$
is called a Jordan block with eigenvalue $\la$.
If $\la \not=0$, then
\begin{equation} \label{J-inv}
\big(-\J_{k}(-\la) \big)^{-1}=\begin{bmatrix}
\la^{-1} & \la^{-2} & \ldots & \la^{-k} \\
0 & \la^{-1} & \ddots & \vdots \\
\vdots & \ddots & \ddots & \la^{-2} \\
0 & \ldots & 0 & \la^{-1}
\end{bmatrix}.
\end{equation}

Consider the Jordan matrix
$$
 \J_\bjjp (\w):=\bigoplus_{a=1}^{\dd} \J_{\bjp_a}(w_a):=\begin{bmatrix}
\J_{\bjp_1}(w_1) & 0 & \ldots & 0\\
0 &\J_{\bjp_2}(w_2)  & \ddots & \vdots \\
\vdots & \ddots & \ddots & 0\\
0 & \ldots & 0 & \J_{\bjp_\dd}(w_\dd)
\end{bmatrix},
$$
which is the direct sum of Jordan blocks.
Define $\mu_{\bjjp}^{\w} \in \gld^*$ by
\begin{equation} \label{muw}
\mu_{\bjjp}^{\w} (e_{ab}) =  -\J_\bjjp (\w)_{a,b} \qquad \mbox{for $a,b=1,\ldots,d$}.
\end{equation}
Here $\J_\bjjp (\w)_{a,b}$ denotes the $(a,b)$ entry of $\J_\bjjp (\w)_{a,b}$.

Let
$$
 \J_{\jjp} (\z)=\bigoplus_{i=1}^{\el} \J_{\jp_i}(z_i)
$$
Define $\nu_{\jjp}^{\z} \in \gls^*$ by
\begin{equation} \label{ovmuz}
\hspace{1cm} \nu_{\jjp}^{\z} (E^i_j)= (-1)^{|i|+1} \J_{\jjp} (\z)_{i,j}  \qquad \mbox{for $i, j \in \I$}.
\end{equation}
Note that $\nu_{\jjp}^{\z}$ vanishes on the odd part of $\gls$.

We write
\begin{equation} \label{Gaudin-Jordan}
\Adzw=\cA_d^{\mu_{\bjjp}^{\w}} (\z, \jjp) \quad \text{and} \quad \Arwz=\cA_\ds^{\nu_{\jjp}^{\z}}(\w, \bjjp).
\end{equation}
We also write
$$
\wLdzw=\wLdmu  \quad \text{and} \quad \Lrwz=\Lr^{\nu_{\jjp}^{\z}}(\w, \bjjp),
$$
where the matrices $\wLdmu$ and $\Lr^{\nu_{\jjp}^{\z}}(\w, \bjjp)$ are defined as in \eqref{wtL} and \eqref{Lsmu}.

Fix $\s:=(0^p, 1^q, 0^m, 1^n) \in \cS_\ds$. As discussed earlier, the Gaudin algebras $\Adzw$ and $\Arwz$ are determined by $\cdet \!  \lb \wLdzw \rb$ and $\Bers \!  \lb \Lrwz \rb$, respectively.

The maps $\vpd$ and $\vps$, given in \propref{gld-map} and \propref{glr-map}, extend naturally to the superalgebra homomorphisms
$$
\vpd: U \big(\gldzb \big) \zpz \longrightarrow \cD \zpz,
$$
and
$$
\vps: U \big(\glsw \big)\zpz \longrightarrow \cD \zpz,
$$
respectively.
For $1 \le i \le \el$, define
\begin{equation} \label{jp}
[\jp_i]=\begin{cases}
\,   \jp_i, &  \mbox{if} \ \  1 \le i \le p^\prime \quad \text{or} \quad  p^\prime+q^\prime+1 \le i \le p^\prime+q^\prime+m^\prime;\\
-\jp_i, &  \mbox{otherwise}.
\end{cases}
\end{equation}

\begin{thm} \label{cdet-Ber}
We have 
$$
 \prod_{i=1}^{\el} (\pz-z_i)^{[\jp_i]}  \cdot \vpd \! \lb  \cdet \!  \lb \wLdzw  \rb \rb = \prod_{a=1}^{\dd} (z-w_a)^{\bjp_a} \cdot  \vps \! \lb  \Bers \!  \lb \Lrwz \rb \rb
$$
as an equality between elements of $\cD \zpz$.
\end{thm}

\begin{proof}
Let
$$
J=\bigoplus_{a=1}^{\dd} \big(\! - \J_{\bjp_a}(w_a-z) \big) \quad \text{and} \quad  J^\prime=\bigoplus_{i=1}^{\el} \big(\! -\J_{\jp_i}(z_i-\pz) \big).
$$
By \eqref{Pmuz}, we find that
\begin{equation} \label{wLdzw}
 \wLdzw =  J^t - \ls \sum_{i=1}^{\el} \sum_{k=0}^{\jp_i-1}  \frac{e_{ab}\otimes \otzi^k}{(\pz-z_i)^{k+1}} \rs^t_{\! a, b=1, \ldots, d}
\end{equation}
and
\begin{equation} \label{Lswz}
\Lrwz= J^\prime- \ls (-1)^{|i|} \sum_{a=1}^{\dd} \sum_{k=0}^{\bjp_a-1}  \frac{E^i_j \otimes \otwa^k}{(z-w_a)^{k+1}} \rs_{\! i,j \in \I}.
\end{equation}

Consider the $(m+n) \times d$ matrices
$$
X:=\big[ (-1)^{|p+q+i|} \xai \big]^{a=1, \ldots, d}_{i=1, \ldots, m+n} \quad \text{and} \quad P_X:=\big[ \pxai \big]^{a=1, \ldots, d}_{i=1, \ldots, m+n}
$$
and the $(p+q) \times d$ matrices
$$
Y:=\big[ (-1)^{|r|+1} \yar \big]^{a=1, \ldots, d}_{r=1, \ldots, p+q} \qquad \text{and} \qquad  P_Y:=\big[  (-1)^{|r|} \pyar \big]^{a=1, \ldots, d}_{r=1, \ldots, p+q}.
$$
Write
$$
M=
\begin{bmatrix}
Y^t \hspace{2mm} P_X^t
\end{bmatrix}
\quad \text{and}  \quad
M^\prime=
\begin{bmatrix} P_Y  \vspace{3mm}  \\ X
\end{bmatrix}.
$$
Let
$$
\fL=
\begin{bmatrix}
J^t & M \vspace{2mm}
 \\ M^\prime & J^\prime
\end{bmatrix}.
$$
It is straightforward to check that $\fL$ is an amply invertible Manin matrix of type $\hat{\s}:=(0^d, \s) \in \cS_{d+p+m|q+n}$.
By \propref{decomp} and \propref{CFR}, we have
$$\Ber^{\! \hat{\s}} \, (\fL)=\cdet (J^t) \cdot \Bers (J^\prime- M^\prime (J^t)^{-1} M).$$
We see that
\begin{eqnarray*}
\vps \! \lb  \Bers \!  \lb \Lrwz \rb \rb
&=&\Bers \lbb J'- \ls (-1)^{|i|} \sum_{a=1}^{\dd} \sum_{k=0}^{\bjp_a-1}  \frac{\vps \big( E^i_j \otimes \otwa^k \big)}{(z-w_a)^{k+1}} \rs_{\! i,j \in \I} \rb\\
&=&\Bers (J'- M^\prime (J^t)^{-1} M).
\end{eqnarray*}
Here the first equality follows from \eqref{Lswz} while the second follows from \propref{glr-map} and \eqref{J-inv}.
Consequently,
$$
\Ber^{\! \hat{\s}} \, (\fL)=  \prod_{a=1}^{\dd} (z-w_a)^{\bjp_a}  \cdot \vps \lb \Bers \! \big( \Lrwz \big) \rb\!.
$$
On the other hand, as $J^t- M (J^\prime)^{-1} M^\prime$ is a sufficiently invertible Manin matrix of type $(0^d)$, \propref{decomp-2} and \propref{CFR} imply that
$$
\Ber^{\! \hat{\s}} \, (\fL) =
\Bers (J^\prime) \cdot \cdet (J^t- M (J^\prime)^{-1} M^\prime).
$$
By \eqref{J-inv}, \eqref{wLdzw} and \propref{gld-map},
\begin{eqnarray*}
\Ber^{\! \hat{\s}} \, (\fL)
&=& \prod_{i=1}^{\el} (\pz-z_i)^{[\jp_i]} \cdot \cdet  \lb  J^t - \ls \sum_{i=1}^{\el} \sum_{k=0}^{\jp_i-1}  \frac{\vpd (e_{a,b}\otimes \otzi^k)}{(\pz-z_i)^{k+1}} \rs^t_{\! a, b=1, \ldots, d} \rb\\
&=& \prod_{i=1}^{\el} (\pz-z_i)^{[\jp_i]} \cdot \vpd \! \lb  \cdet \!  \lb \wLdzw  \rb \rb.
\end{eqnarray*}
This proves the theorem.
\end{proof}

We have the following duality between the Gaudin algebras $\Adzw$ and $\Arwz$. We call it the \emph{duality of $(\gld, \gls)$ for Gaudin models with irregular singularities}.

\begin{thm}\label{Gaudin-duality}
$
\vpd \big( \Adzw \big)=\vps \big(\Arwz \big).
$
\end{thm}
\begin{proof}
This follows from \propref{L'} and \thmref{cdet-Ber}.
\end{proof}

The Weyl superalgebra $\cD$ is a subalgebra of $\End(\F)$. The maps $\vpd$ and $\vps$ induce actions of $\Adzw$ and $\Arwz$ on $\F$. \thmref{Gaudin-duality} also says that these actions are equivalent.

If any of $p,q,m,n$ is set to 0, \thmref{Gaudin-duality} will give a special version of the duality. 
Taking $p=q=n=0$, \thmref{cdet-Ber} gives the identity in \cite[Theorem 4.8]{VY}, and we recover the duality of $(\gld, \gl_m)$ due to Vicedo and Young.

\begin{cor} [{\cite[Theorem 4.8]{VY}}] \label{VY}
Set $p=q=n=0$. Then $\vpd \big( \Adzw \big)=\vps  \lb \cA_{m}^{\z,  \jjp}(\w, \bjjp) \rb$.
\end{cor}

\thmref{Gaudin-duality} also gives a fermionic realization of the above duality.

\begin{cor} \label{fermi}
Set $p=q=m=0$. Then $\vpd \big( \Adzw \big)=\vps  \lb \cA_{0|n}^{\z,  \jjp}(\w, \bjjp) \rb$.
\end{cor}

Taking $q=n=0$ in \thmref{Gaudin-duality}, we obtain the duality of $(\gld, \gl_{m+p})$ for the bosonic oscillators in which the Fock space $\F$ decomposes into a direct sum of tensor products of infinite-dimensional $\gl_{m+p}$-modules.

\begin{cor} \label{bosonic}
Set $q=n=0$. Then $\vpd \big( \Adzw \big)=\vps  \lb \cA_{m+p}^{\z,  \jjp}(\w, \bjjp) \rb$.
\end{cor}

Further, the following duality holds.

\begin{cor} \label{variant}
Suppose $\dd=d$ and $\el=p+q+m+n$. Then
$\vpd \! \lb  \cA_d^{\w \!, \, (1^\dd)} \lb \z,  (1^\el) \rb \rb=\vps \! \lb  \cA_\ds^{\z,  (1^\el) } \lb \w, (1^\dd) \rb \rb$.
\end{cor}

\begin{rem}
A variant of \corref{variant}, called the Bethe duality of $(\gld, \gls)$, is established in \cite{ChL25-2}.
\end{rem}

\subsection{An application of the duality} \label{app}

In this subsection, we consider the actions of $\Adzw$ and $\Arwz$ on the Fock space $\F$ via the maps $\vpd$ and $\vps$, respectively; see \propref{gld-map} and \propref{glr-map}.
We will restrict our attention to $\jjp=(1^\el)$, where $\el=p+q+m+n$.

First of all, we review some well-known properties of the action of the Gaudin algebras (with the simplest possible singularities) on finite-dimensional $\gld$-modules.
An element $\mu \in \gld^*$ is called \emph{regular} if the dimension of the centralizer $\gld^\mu$ of $\mu$ in $\gld$ is $d$.
For $\ell \in \N$, let
$$
\Xl=\setc*{\! (z_1, \ldots,z_\ell)\in \C^\ell}{z_i\not=z_j \,\, \mbox{for any $i\not=j$} \!}
$$
be the \emph{configuration space} of $\ell$ distinct points on $\C^\ell$.
The following is a consequence of \cite{FFRy}.

 \begin{prop}   \label{cyclic+diag}
Let $\bn V= V_1 \otimes \cdots \otimes V_\el$, where $V_1, \ldots, V_\el$ are finite-dimensional irreducible $\gld$-modules.
Then $\bn V$ is a cyclic $\cA_d^{\w, \bjjp}(\z, (1^\el))$-module for $\w \in \Xd$ and $\z \in \Xpqmn$.
 Moreover, the Gaudin algebra $\cA_d^{\w, \bjjp}(\z, (1^\el))_{\bn V}$ is diagonalizable with a simple spectrum for generic $\w \in \Xd$ and $\z \in \Xpqmn$.
\end{prop}

\begin{proof}
The first statement follows from \cite[Corollary 5]{FFRy} since the element $\mu_{\bjjp}^{\w} \in \gld^*$, defined in \eqref{muw}, is regular.
The second statement follows from the first one, \cite[Lemma 2]{FFRy}, and the property that having a simple spectrum is an open condition.
\end{proof}

Every $\glsw$-module is a $\gls$-module via the diagonal map from $\gls$ to $\glsw$. For any $\mu\in \fh_{\ds}^{*}$ and any $\glsw$-module $M$, let
\begin{equation} \label{weight}
M_\mu=\setc*{\! v \in M }{ h v = \mu(h) v  \quad  \mbox{for all $h \in \fh_{\ds}$} \!}.
\end{equation}
If $M_\mu \not=0$, then $\mu$ is called a weight of $M$ and $M_\mu$ is called the $\mu$-weight space of $M$. 
There is a weight space decomposition
\begin{equation} \label{Fmu}
\F=  \bigoplus_{\mu\in \fh_{\ds}^{*}}  \F_\mu.
\end{equation}

Let $\jjp=(1^\el)$, where $\el=p+q+m+n$.
For $1\leq i\leq \el$, define $\gld^{(i)}=\gld(z_i, 1)\cong \gld$. 
Let $W^{(i)}_{\mu_{i}}$ be the subspace of $\C[y^1_i,\cdots, y^d_i]$ (resp.,  $\C[x^1_{i-p-q},\cdots, x^d_{i-p-q}]$) spanned by monomials of degree $\mu_i\in-\Zp$ (resp.,  $\Zp$) if $1 \le i \le p+q$ (resp., $p+q+1 \le i \le \el$). 
We have direct sum decompositions of $\gld^{(i)}$-modules: 
$$
\C[y^1_i,\cdots, y^d_i]=\bigoplus_{\mu_i\in-\Zp} W^{(i)}_{\mu_{i}}, \qquad 1 \le i \le p+q,
$$
and
$$
 \C[x^1_{i-p-q},\cdots,x^d_{i-p-q}]=\bigoplus_{\mu_i \in \Zp} W^{(i)}_{\mu_{i}},  \qquad p+q+1 \le i \le \el.
$$
Evidently, $W^{(i)}_{\mu_{i}}$ is finite-dimensional, and it is well known that $W^{(i)}_{\mu_{i}}$ is an irreducible $\gld^{(i)}$-module for each $i$ (see, for instance, \cite[Theorem 3.3]{CLZ} by setting exactly one of $p,q,m,n$ (given there) to 1 and the others to 0). 

For $\mu_1, \ldots, \mu_{p+q}\in -\Zp$ and $\mu_{p+q+1},\ldots,\mu_\el \in \Zp$, define
$$
\Wmu= W_{\mu_1}^{(1)} \otimes \cdots \otimes W_{\mu_\el}^{(\el)}.
$$
The Fock space $\F$ decomposes into a direct sum of $\gld(\z, (1^\el))$-modules
\begin{equation} \label{Fmub}
\F=\bigoplus_{\substack{\mu_1, \ldots, \mu_{p+q}\in -\Zp,\\ \mu_{p+q+1},\ldots,\mu_\el \in \Zp}} \Wmu,
\end{equation}
and each $\Wmu$ is a $\gld$-module via the diagonal map from $\gld$ to $\gld(\z, (1^\el))$. 

For any $\gld$-module $M$ and $\mu \in \fh_d^*$, the space $M_\mu$ is defined as in \eqref{weight} by replacing $\fh_{\ds}$ with $\fh_d$. 
The action of $\fh_d$ on $\F$ is determined by the action of $e_{aa}$ on $\F$, which is given by
$$
-\sum_{r=1}^{p+q} y^a_r \pa_{y^a_r}+\sum_{i=1}^{m+n} (-1)^{|p+q+i|}  \pa_{x^a_i} x^a_i =-\sum_{r=1}^{p+q} y^a_r \pa_{y^a_r}+\sum_{i=1}^{m+n} x^a_i \pa_{x^a_i}+(m-n),
$$
for $1 \le a \le d$.
For $k_1, \ldots k_{\dd}  \in \Z$, we define 
$$
\F_{[k_1,\ldots, k_\dd]}=\bigoplus_{\eta}  \F_\eta,
$$
where the direct sum is taken over all $\eta\in \fh_d^*$ such that
$\sum_{i=d_a+1}^{d_{a+1}} \eta(e_{ii})=k_a+ (m-n) \bjp_a$
for $1 \le a \le \dd$. 
Thus, $\F$ has a direct sum decomposition of subspaces
\begin{equation} \label{Fk}
\F=\bigoplus_{k_1, \ldots, k_{\dd} \in \Z} \F_{[k_1, \ldots, k_{\dd}]}.
\end{equation}
For $\mu_1, \ldots, \mu_{p+q}\in -\Zp$ and $\mu_{p+q+1},\ldots,\mu_\el \in \Zp$, let
$$
 \Wmu_{[k_1, \ldots, k_{\dd}]}= \Wmu \cap \F_{[k_1, \ldots, k_{\dd}]}.
$$
Then
\begin{equation} \label{Wmu}
\Wmu=\bigoplus_{k_1, \ldots, k_{\dd} \in \Z} \Wmu_{[k_1, \ldots, k_{\dd}]}.
\end{equation}

For $k_1, \ldots k_{\dd}  \in \Z$, recall that $\Vk =V^{(\xi_1)}_{k_1}  \otimes \cdots \otimes V^{(\xi_{\dd})}_{k_{\dd}}$, where $V^{(\xi_a)}_{k_a}$ is the $\gls(w_a, \xi_a)$-submodule of $\C [\Sig^{(a)} ]$ spanned by monomials of degree $k_a$ (see \secref{ex}). Note that $\Vk$ is infinite-dimensional if $p \not= 0$ and $m \not= 0$.

\begin{prop} \label{wt=wt}
For any $k_1, \ldots k_{\dd} \in \Z$ and any weight $\mu$ of $\Vk$, we have 
$$
\Vk_\mu =\Wmubk,
$$  
where 
$$
\overline{\mu}_r=\begin{cases}
\, \mu(E^r_r)+(-1)^{|r|} d &   \mbox{if} \ \ 1\le r \le p+q ;\\
\, \mu(E^r_r) \ \ & \mbox{if} \ \ p+q+1\le r \le \el.
\end{cases}
$$
\end{prop}

\begin{proof} 
Let $k_1, \ldots k_{\dd} \in \Z$.
By \eqref{F-decomp} and \eqref{Fk}, we have
$$
 \F_{[k_1,\ldots, k_\dd]}=\Vk.
$$
The Cartan subalgebra $\fh_\ds$ of $\gls$ acts on $\F$ as follows: For $1 \le r \le p+q$, $E^r_r$ acts on $\F$ by 
$$
\sum_{a=1}^{\dd} \sum_{\al=\va_a+1}^{\va_{a+1}}  (-1)^{|r|+1}  \pa_{y^{\al}_r} y^\al_r=-\sum_{\al=1}^{d} y^\al_r  \pa_{y^\al_r}+(-1)^{|r|+1} d, 
$$
and for $1 \le i \le m+n$, $E^{p+q+i}_{p+q+i}$ acts on $\F$ by
$$
\sum_{a=1}^{\dd} \sum_{\al=\va_a+1}^{\va_{a+1}}   x^{\al}_i \pa_{x^\al_i}=\sum_{\al=1}^{d} x^{\al}_i \pa_{x^\al_i}.
$$
For any weight $\mu$ of $\Vk$, the above descriptions together with \eqref{Fmu} and \eqref{Fmub} give 
$$
\F_\mu=\Wmub.
$$
Hence
$$
\Vk_\mu=\F_{[k_1,\ldots, k_\dd]}\cap\F_\mu=\Wmubk,
$$
as desired.
\end{proof}

According to the duality for Gaudin models with irregular singularities, we expect that an analog of \propref{cyclic+diag} should be valid for the action of the Gaudin algebra $\cA_\ds^{\z,  (1^{\el}) }(\w, \bjjp)$ on some $\glsw$-modules.
The following theorem supports our expectation.

\begin{thm} \label{Gaudin-app}
For any $k_1, \ldots k_{\dd} \in \Z$ and any weight $\mu$ of $\Vk$, the $\mu$-weight space $\Vk_\mu$ is a cyclic $ \cA_\ds^{\z,  (1^{\el}) }(\w, \bjjp)$-module for $\w \in \Xd$ and $\z \in \Xpqmn$.
Moreover, the Gaudin algebra $\cA_\ds^{\z,  (1^{\el}) }(\w, \bjjp)_{ \Vk _{\! \mu}}$ is diagonalizable with a simple spectrum for generic $\w$ and $\z$.
\end{thm}

\begin{proof}

Let $k_1, \ldots k_{\dd} \in \Z$, and let $\mu$ be any weight of $\Vk$. The Gaudin algebra $\cA_\ds^{\z,  (1^{\el}) }(\w, \bjjp)$ commutes with the diagonal action of the centralizer $\gls^{\nu}$ of $\nu$ in $\gls$ (see \cite[Proposition 4]{Ry06}), where $\nu:=\nu_{(1^\el)}^{\z}$ is defined as in \eqref{ovmuz}. Since $\nu$ corresponds to a diagonal matrix, $\gls^{\nu}$ contains $\fh_\ds$. Thus, $\cA_\ds^{\z,  (1^{\el}) }(\w, \bjjp)$ preserves weight spaces, and the weight space $\Vk_\mu$ is an $\cA_\ds^{\z,  (1^{\el}) }(\w, \bjjp)$-module.
By \thmref{Gaudin-duality} and \propref{wt=wt}, $\Wmubk$ is an $\cA_d^{\w, \bjjp}(\z, 1^\ell)$-module, and
\begin{equation} \label{eq-act}
\cA_\ds^{\z,  (1^{\el}) }(\w, \bjjp)_{\Vk_\mu}=\cA_d^{\w, \bjjp}(\z, (1^\el))_{\Wmubk}.
\end{equation}
As in \eqref{Wmu}, $\Wmub$ is a direct sum decomposition of $\cA_d^{\w, \bjjp}(\z, (1^\el))$-modules.
We apply \propref{cyclic+diag} to the decomposition of $\Wmub$ to see that the space $\Wmubk$ is a cyclic $\cA_d^{\w, \bjjp}(\z, (1^\el))$-module for $\w \in \Xd$ and $\z \in \Xpqmn$, and $\cA_d^{\w, \bjjp}(\z, (1^\el))_{\Wmubk}$ is diagonalizable with a simple spectrum for generic $\w$ and $\z$.
In view of \eqref{eq-act}, the theorem follows.
\end{proof}

\section{A duality for classical Gaudin models} \label{classical}

This section is devoted to introducing the classical Gaudin algebras with irregular singularities for $\gld$ and $\gl_\pqmn$ and establishing a duality for classical Gaudin models.

\subsection{Settings} \label{settings}

For any Lie superalgebra $\fg$,  the universal enveloping algebra $U(\fg)$ of $\fg$ has a filtration
\begin{equation} \label{filtration-g}
\C=U_0(\fg)\subset U_1(\fg)\subset U_2(\fg)\subset\cdots U_i(\fg)\subset \cdots
\end{equation}
such that $\bigcup_{i=0}^\infty U_i(\fg)=U(\fg)$ and $U_i(\fg)  U_{j}(\fg)\subseteq U_{i+j}(\fg)$ for all $i, j \in \Zp$,
where $U_i(\fg)$ is the subspace of $U(\fg)$ spanned by all products $A_1 \ldots A_j$, for $j \le i$ and $A_k \in \fg$.
We call \eqref{filtration-g} the \emph{canonical filtration} on $U(\fg)$.

Let ${\rm gr}(U(\fg))=\bigoplus_{i=0}^\infty U_{i}(\fg)/U_{i-1}(\fg)$ (where $U_{-1}(\fg):=0$). It is called the \emph{associated graded algebra} of $U(\fg)$ with respect to the canonical filtration.
Let $\cS(\fg)$ denote the \emph{supersymmetric algebra} of $\fg$. 
We have the following superalgebra isomorphism (see \cite[Proposition 2.3.6]{Dix}): 
\begin{equation} \label{grU} 
{\rm gr} \! \lb U(\fg) \rb \cong \cS(\fg).
\end{equation}
Because of \eqref{grU}, we say that $U(\fg)$ is a quantization of $\cS(\fg)$.

A \emph{Poisson superalgebra} $\cA$ is a supercommutative superalgebra over $\C$ equipped with a $\C$-bilinear map $\{\cdot,  \cdot\}: \cA\times \cA\longrightarrow \cA$, called a Poisson bracket, such that $(\cA,\{\cdot, \cdot\})$ is a Lie superalgebra satisfying the Leibniz rule, i.e.,
$$
\{A, BC\}=\{A, B\} C+ (-1)^{|A||B|}B \{A, C\}
$$
for any homogeneous elements $A, B, C \in \cA$.

The superalgebra ${\rm gr} \! \lb U(\fg) \rb$ is naturally a Poisson superalgebra (the construction in the proof of \cite[Proposition 1.3.2]{CG} works here by taking into account the sign rule). This induces a Poisson superalgebra structure on $\cS(\fg)$ via \eqref{grU}.

Let $\fg$ be a general linear Lie (super)algebra.
Fix $\ell \in \N$, $\z \in \C^\ell$ and $\jjp \in \N^\ell$.
For even variables $t$ and $z$, the map $\Psi_{(\z, \jjp)}:=\Psi^{0}_{(\z, \jjp)}$, defined in \secref{FFC}, preserves filtrations and hence descends to the Poisson superalgebra homomorphism
$$\Phz: \cS(\fgm) \longrightarrow \cS \big(\gz \big) [\tns[z^{-1}]\tns].$$
We add $z$ to the symbol $\Phz$ to emphasize the dependency on $z$.
Also,
\begin{equation} \label{Phmuz}
\Phz \lb A \otimes t^{-1} \rb= - \sum_{i=1}^{\ell}  \sum_{k=0}^{\jp_i-1} \frac{A \otimes \otzi^k}{(z-z_i)^{k+1}}, \qquad \mbox{for $A \in \fg$.}
\end{equation}

We will introduce the classical Gaudin algebras for $\gld$ and $\gls$ separately.
From now on, we fix two commuting even variables $z$ and $w$.

\subsection{Classical Gaudin algebras for $\gld$} \label{cl-gld}

Recall the definition of the determinant given in \secref{Ber}.
Let $\z \in \C^\ell$ and $\jjp \in \N^\ell$.
For any $\mu \in \gld^*$, we consider the matrix
$$
\ov \cL_d^{\mu}(\z, \jjp)  = \Big[\delta_{i,j} w + \Phz \lb e_{ij}\otimes t^{-1} \rb +\mu(e_{ij})  \Big]_{i,j=1,\ldots,d}
$$
over the commutative algebra $\cS \big(\gldz \big) [\tns[z^{-1}]\tns] [w]$.
We have an expansion
$$
\det(\ov \cL_d^{\mu}(\z, \jjp) )=w^d+\sum_{i=1}^{d} \ov{a}_i (z)w^{i-1}
$$
for some $\ov{a}_i (z) \in \cS \big(\gldz \big) [\tns[z^{-1}]\tns]$.
Let $\ov\cA_d^{\mu}(\z, \jjp)$ be the subalgebra of $\cS \big(\gldz \big)$ generated by the coefficients of the series $\ov{a}_i (z)$ for $i=1, \ldots, d$.
It is called the \emph{classical Gaudin algebra for $\gld$ with singularities of orders $\jp_i$ at $z_i$, $i=1, \ldots, \ell$, relative to $\mu$}, or simply the \emph{classical Gaudin algebra corresponding to $\cA_d^{\mu}(\z, \jjp)$}. The algebra $\ov\cA_d^{\mu}(\z, \jjp)$ is known to be Poisson-commutative \cite{FFTL}.

Let ${\rm gr}\! \lb \cA_d^{\mu}(\z, \jjp) \rb$ be the associated graded algebra of $\cA_d^{\mu}(\z, \jjp)$ with respect the filtration induced by the canonical filtration on $U(\gldz)$.
We remark that the proof of \cite[Theorem 3.4]{FFTL} shows that $\cA_d^{0}(\z, \jjp)$ is a quantization of $\ov\cA_d^{0}(\z, \jjp)$, i.e.,
\begin{equation} \label{FFTL}
{\rm gr}\! \lb \cA_d^{0}(\z, \jjp) \rb \cong \ov\cA_d^{0}(\z, \jjp).
\end{equation}
For $\z=(0)$ and $\jjp=(1)$, the algebra $\ov\cA_d^\mu:=\ov\cA_d^{\mu}((0), (1))$ is simply the \emph{shift of argument subalgebra} of $\cS(\fg)$ introduced by Mishchenko and Fomenko; see \cite{MF}.
Let $\cA_d^\mu=\cA_d^{\mu}((0), (1))$. Futorny and Molev \cite[Main Theorem]{FM} show that
\begin{equation} \label{FM}
\hspace{1cm} {\rm gr} (\cA_d^\mu) \cong \ov\cA_d^\mu \quad \qquad \mbox{for any $\mu \in \gld^*$.}
\end{equation}
We expect a more general statement to hold; see \conjref{gr} below.

\subsection{Classical Gaudin algebras for $\gl_\pqmn$} \label{cl-glr}

Recall the symbol $\ds=\pqmn$.
Let $\z \in \C^\ell$ and $\jjp \in \N^\ell$.
For any $\mu \in \gls^*$ which vanishes on the odd part of $\gls$, let
$$
\rLs^{\mu}(\z, \jjp) = \Big[ \delta_{i,j}z+ (-1)^{|i|} \lb \Phw   \lb E^i_j \otimes t^{-1} \rb +\mu(E^i_j) \rb  \Big]_{i,j\in \I},
$$
which is an amply invertible $(p+q+m+n) \times (p+q+m+n)$ matrix of type $\s:=(0^p, 1^q, 0^m, 1^n)$ over the supercommutative superalgebra $\cS \big(\glsz \big) \blb z^{-1}, w^{-1} \brb$.
We have an expansion 
$$
\Bers \! \left( \rLs^{\mu}(\z, \jjp)  \rb=\sum_{i=-\infty}^{p+m-q-n} \ov{b}_i (w) z^i
$$
for some $\ov{b}_i (w) \in \cS \big(\glsz \big) \blb w^{-1} \brb$.
Let $\ov\cA_\ds^{\mu}(\z, \jjp)$ be the subalgebra of $S \big(\glsz\big)$ generated by the coefficients of the series $\ov{b}_i (w)$, for $i \in \Z$ with $i \le p+m-q-n$.
It is called the \emph{classical Gaudin algebra for $\gls$ with singularities of orders $\jp_i$ at $z_i$, $i=1, \ldots, \ell$, relative to $\mu$}, or simply the \emph{classical Gaudin algebra corresponding to $\cA_\ds^{\mu}(\z, \jjp)$}.
It is Poisson-commutative by an argument parallel to the proof of \cite[Corollary 3.6]{MR}.
If $q=n=0$ and $d=p+m$, the algebra coincides with the one given in \secref{cl-gld}.

Let ${\rm gr}\! \lb \cA_\ds^{\mu}(\z, \jjp) \rb$ be the associated graded algebra of $\cA_\ds^{\mu}(\z, \jjp)$ with respect to the filtration induced by the canonical filtration on $U(\glsw)$.
We conjecture that the Gaudin algebras are quantizations of the classical Gaudin algebras.

\begin{conj} \label{gr}

${\rm gr} \lb \cA_\ds^{\mu}(\z, \jjp) \rb \cong \ov\cA_\ds^{\mu}(\z, \jjp)$ for any $\mu \in \gls^*$ which vanishes on the odd part of $\gls$.

\end{conj}

The conjectures hold for some special cases (see \eqref{FFTL} and \eqref{FM}), but the general case is wide open.

\subsection{A classical duality of $(\gld, \gl_\pqmn)$} \label{classical-duality}

Let $\dd$, $\el$, $\w$, $\z$, $\bjjp$ and $\jjp$ be as in \secref{Fock}.
In \thmref{Gaudin-duality}, we establish the duality of $(\gld, \gls)$ for Gaudin models with irregular singularities, which is an equivalence between the actions of the Gaudin algebras $\Adzw$ and $\Arwz$ on the Fock space $\F$. 
The duality and \conjref{gr} lead us to expect that there should be a duality between the classical Gaudin algebras corresponding to $\Adzw$ and $\Arwz$. 
In this subsection, we will see that the expectation is confirmed.

Recall the Weyl superalgebra $\cD$ defined in \secref{Fock}. We define the degree of any monomial in $\cD$ in the usual way. In particular, the degrees of the generators $\xai$, $\yar$, $\pxai$ and $\pyar$ of $\cD$ are defined to be $1$, for $1 \le i \le m+n$, $1 \le r \le  p+q$ and $1 \le a \le  d$.
For $k \in \N$, let $\cD_k$ be the subspace of $\cD$ spanned by all monomials of degree less than or equal to $k$ in the generators $\xai$, $\yar$, $\pxai$ and $\pyar$, $1 \le i \le m+n$, $1 \le r \le  p+q$ and $1 \le a \le  d$, and let $\cD_0=\C$.
Then
\begin{equation}\label{filtration-D}
\cD_0\subset \cD_1\subset \cD_2\subset\cdots \cD_k \subset \cdots
\end{equation}
is a filtration on $\cD$ such that $\bigcup_{k=0}^\infty \cD_k =\cD$ and $\cD_k  \cD_{l}\subseteq \cD_{k+l}$ for all $k, l  \in \Zp$.
Just as \eqref{filtration-g}, we call \eqref{filtration-D} the canonical filtration on $\cD$.

Let ${\rm gr(\cD)}=\bigoplus_{k=0}^\infty \cD_k/\cD_{k-1}$ (where $\cD_{-1}:=0$), called the associated graded algebra of $\cD$ with respect to the canonical filtration, and let $\oD$ denote the polynomial superalgebra generated by the variables $\xai$, $\yar$, $\ppxai$ and $\ppyar$, for $1 \le i \le m+n$, $1 \le r \le  p+q$ and $1 \le a \le  d$, where $\xai$ and $\ppxai$ (resp., $\yar$ and $\ppyar$) are even for $1 \le i \le m$ (resp., $1 \le r \le p$) and are odd otherwise.  
It is well known that there is a superalgebra isomorphism
\begin{equation} \label{grD}
{\rm gr} \! \lb \cD \rb \cong \oD
\end{equation}
sending the elements $\xai+\cD_0$, $\yar+\cD_0$, $\pxai+\cD_0$ and $\pyar+\cD_0$ in $\cD_1/\cD_0$, respectively, to the elements $\xai$, $\yar$, $\ppxai$ and $\ppyar$ in $\oD$, for $1 \le i \le m+n$, $1\le r \le  p+q$ and $1 \le a \le  d$ (see, for example, \cite[Example 1.3.3]{CG} and the subsequent discussion). 
Moreover, the superalgebra ${\rm gr} \! \lb \cD \rb$ is naturally a Poisson superalgebra (which is seen as in the case of ${\rm gr} \! \lb U(\fg) \rb$ in \secref{settings}),
and thus we have a Poisson superalgebra structure on $\oD$ via \eqref{grD}.
More explicitly, $\oD$ is a Poisson superalgebra with the Poisson bracket $\{ \cdot, \cdot \}$ given by
\begin{equation} \label{Pb1}
\{ \xai, \xbj \}= \{ \xai, \ybs \}= \{ \yar, \ybs \}=\{ \ppxai, \ppxbj \}=\{ \ppxai, \ppybs \}= \{ \ppyar, \ppybs \}=0,
\end{equation}
\begin{equation} \label{Pb2}
\{ \ppxai, x^b_j \}=\de_{i,j} \de_{a,b},  \quad \{ \ppxai, y^b_s \}=\{ \ppyar, \xbj \}=0, \quad \{ \ppyar, y^b_s \}=\de_{r,s} \de_{a,b},
\end{equation}
for $1 \le i,j \le m+n$, $1 \le r,s \le  p+q$ and $1 \le a, b \le  d$.

The following proposition follows by descending the maps $\vpd$ and $\vps$, given in \propref{gld-map} and \propref{glr-map}, to the maps on the associated graded algebras since $\vpd (U_i(\gldzb)) \subseteq\cD_{2i}$ and $\vps (U_i(\glsw)) \subseteq\cD_{2i}$ for all $i \in \Zp$. 

\begin{prop} \label{map-classical}
\sloppy
\begin{enumerate} [\normalfont(i)]

\item The superalgebra homomorphism $\vpd : U \big(\gldzb \big) \longrightarrow \cD$, given in \propref{gld-map}, descends to a Poisson superalgebra homomorphism $\ovpd : \cS \big(\gldzb \big) \longrightarrow \oD$.

\item The superalgebra homomorphism $\vps : U \big(\glsw \big) \longrightarrow \cD$, given in \propref{glr-map}, descends to a Poisson superalgebra homomorphism $\ovps : \cS \big(\glsw \big) \longrightarrow \oD$.

\end{enumerate}
\end{prop}

The maps $\ovpd$ and $\ovps$ extend naturally to the Poisson superalgebra homomorphisms
$$
\ovpd: \cS \big(\gldzb \big) \blb z^{-1}, w^{-1} \brb \longrightarrow \oD \blb z^{-1}, w^{-1} \brb ,
$$
and
$$
\ovps: \cS \big(\glsw \big) \blb z^{-1}, w^{-1} \brb  \longrightarrow \oD  \blb z^{-1}, w^{-1} \brb,
$$
respectively.
Let $\mu_{\bjjp}^{\w} \in \gld^*$ and $\nu_{\jjp}^{\z} \in \gls^*$ be as in \secref{gld-gls-duality}.
We write
$$
\Cdzw=\ov\cA_d^{\mu_{\bjjp}^{\w}} (\z, \jjp) \quad \text{and} \quad \Crwz=\ov\cA_\ds^{\nu_{\jjp}^{\z}}(\w, \bjjp).
$$
We also write
$$
\rLdzw=\ov \cL_d^{\mu_{\bjjp}^{\w}} (\z, \jjp) \quad \text{and} \quad \rLswz=\rLs^{\nu_{\jjp}^{\z}}(\w, \bjjp).
$$
The following theorem gives a duality between the classical Gaudin algebras $\Cdzw$ and $\Crwz$, which we call the \emph{duality of $(\gld, \gls)$ for classical Gaudin models with irregular singularities}.

\begin{thm}\label{class-dual}
Recall $[\jp_i]$, $1 \le i \le \el$, defined in \eqref{jp}.
We have
$$
 \prod_{i=1}^{\el} (z-z_i)^{[\jp_i]}  \cdot \ovpd \! \lb \det \!  \lb \rLdzw  \rb \rb = \prod_{a=1}^{\dd} (w-w_a)^{\bjp_a} \cdot  \ovps \! \lb  \Bers \!  \lb \rLswz \rb \rb.
$$
Consequently,
$\ovpd \big( \Cdzw \big)=\ovps \big(\Crwz \big).$

\end{thm}

\begin{proof}
By \eqref{Phmuz},
$$
\rLdzw= \bigoplus_{a=1}^{\dd} \big(\! - \J_{\bjp_a}(w_a-w) \big)- \ls \sum_{i=1}^{\el} \sum_{k=0}^{\jp_i-1}  \frac{e_{ab}\otimes \otzi^k}{(z-z_i)^{k+1}} \rs_{\! a, b=1, \ldots, d}
$$
and
$$
\rLswz=\bigoplus_{i=1}^{\el}\big(\! -  \J_{\jp_i}(z_i-z) \big)- \ls (-1)^{|i|} \sum_{a=1}^{\dd} \sum_{k=0}^{\bjp_a-1}  \frac{E^i_j \otimes \otwa^k}{(w-w_a)^{k+1}} \rs_{\! i,j \in \I}.
$$
By the same argument as in \thmref{cdet-Ber}, we see that
$$
 \prod_{i=1}^{\el} (z-z_i)^{[\jp_i]}  \cdot \ovpd \! \lb \det\!  \lb \rLdzw^t  \rb \rb =  \prod_{a=1}^{\dd} (w-w_a)^{\bjp_a} \cdot  \ovps \! \lb  \Bers \!  \lb \rLswz \rb \rb \!.
$$
Since $\det\!  \lb \rLdzw^t  \rb=\det\!  \lb \rLdzw  \rb$, the theorem follows.
\end{proof}

If any of $p,q,m,n$ is set to 0, we will obtain a special version of the duality. 
Taking $p=q=n=0$, the identity in \thmref{class-dual} recovers \cite[Theorem 3.2]{VY}. Taking $p=q=m=0$, it recovers \cite[Theorem 3.4]{VY} by \propref{Ber1n}.
In other words, \thmref{class-dual} yields the following corollaries.

\begin{cor} [{\cite[Theorem 3.2]{VY}}]
Suppose $p=q=n=0$. Then $\ovpd \big( \Cdzw \big)=\ovps \big(  \ov\cA_{m}^{ \z,  \jjp}(\w, \bjjp) \big)$.
\end{cor}

\begin{cor} [{\cite[Theorem 3.4]{VY}}]
Suppose $p=q=m=0$. Then $\ovpd \big( \Cdzw \big)=\ovps \big(  \ov\cA_{0|n}^{ \z,  \jjp}(\w, \bjjp) \big)$.
\end{cor}

\begin{cor}
Set $q=n=0$. Then $\ovpd \big( \Cdzw \big)=\ovps \big(  \ov\cA_{p+m}^{ \z,  \jjp}(\w, \bjjp) \big)$.
\end{cor}

\vskip 0.5cm
\noindent{\bf Acknowledgments.}
The first author was partially supported by NSTC grant 113-2115-M-006-010.
The second author was partially supported by NSTC grant 112-2115-M-006-015-MY2.
The authors would like to thank the referee for valuable comments and suggestions.

\end{document}